\documentclass{amsart}

\usepackage{stmaryrd}
\usepackage{lmodern}
\usepackage{diagbox}
\usepackage{mathtools}
\usepackage{amssymb,amsmath}
\usepackage{graphicx}

\usepackage{mathrsfs}

\usepackage{color}
\usepackage{amsfonts}
\usepackage{epsfig}
\usepackage{graphics}

\usepackage{multicol}
\usepackage{lscape}
\usepackage{latexsym}

\def\Fq{{\mathbb F}_q}

\def\Zq-1{{\mathbb Z}_{q-1}}
\def\C{\mathbb C}

\def\F{\mathbb F}

\def\alp{{\alpha}}
\def\bet{{\beta}}

\def\l1z{{\lambda}_{1,z}}
\def\qed{~~\vrule height8pt width4pt depth0pt}

\newtheorem{prop}{Proposition}
\newtheorem{definition}{Definition}
\newtheorem{lemma}{Lemma}
\newtheorem{thm}{Theorem}
\newtheorem{cor}{Corollary}

\newcommand{\be}{\begin{equation}}
\newcommand{\ee}{\end{equation}}
\newcommand{\beq}{\begin{eqnarray}}
\newcommand{\eeq}{\end{eqnarray}}
\newcommand{\beqn}{\begin{eqnarray*}}
\newcommand{\eeqn}{\end{eqnarray*}}

\newcommand{\fq}{{\mathbb F}_{q}}
\newcommand{\fqn}{{\mathbb F}_{q^n}}

\begin{document}

\title[Enumeration of polynomials]{On the enumeration of polynomials with prescribed factorization pattern}

\author{ Simon Kuttner  and Qiang Wang}
\thanks{This research is partially supported by NSERC of Canada (RGPIN-2017-06410).}
\address{
School of Mathematics and Statistics\\
Carleton University\\
Ottawa, Ontario\\
Canada K1S5B6
}
\email{SimonKuttner@cmail.carleton.ca, wang@math.carleton.ca}


\date{\today}

\maketitle

\begin{abstract}
We use generating functions over group rings  to  count polynomials over finite fields with the first few  coefficients prescribed and a factorization pattern prescribed. In particular, we obtain different exact formulas for the number of monic $n$-smooth polynomial of degree $m$ over a finite field,  as well as the number of monic $n$-smooth polynomial of degree $m$ with the prescribed trace coefficient. 

\end{abstract}

\section{Introduction} Let $p$ and $e$ be positive integers where $p$ is prime. Let $\Fq$ be a finite field with $q=p^e$ elements. Let $\Fq[x]$ denote the set of polynomials of polynomials over $\Fq$. Let $M$ denote the set of monic polynomials  over $\Fq$, and $I$ denote the subset of these polynomials that are irreducible. For each positive integer $d$, let $M_d$ denote the set of degree $d$ monic polynomials over $\Fq$, and let $I_d$ be the subset of these polynomials that are irreducible.

For a monic polynomial $f$, let $d(f)$ denote the degree of $f$,  $r_i(f)$ denote the number of monic distinct irreducible factors of $f$ with degree $i$, and $l_i(f)$ denote the number of monic degree $i$ irreducible factors of $f$ counting multiplicity. In particular,  we write
\begin{align}\label{Eq1}
f(x)=x^{d(f)}+f_1x^{d(f)-1}+\cdots+f_{d(f)},
\end{align}
and set $f_j=0$ if $j>d(f)$.

For $f\in M$ and $w\ge 0$,  we define
\begin{align}
\langle f \rangle_w&=x^{d(f)}f(1/x) \pmod {x^{w+1}}\\ \nonumber
&=1+f_1x+\cdots+f_{w}x^{w} \pmod {x^{w+1}}.
\end{align}
In this paper, we want to determine the number of monic polynomials $f(x)$  of degree $m$  with prescribed coefficients $f_1,\ldots,f_w$  and a given pattern of irreducible factors in terms of their degrees.   First of all, we introduce the following definitions.


\begin{definition} \label{genDef} 
Let $w\ge 0$ be a fixed integer and $T\subset\mathbb{N}$ be a finite set.\\
(1) Define $N(m,\prod_{i\in T}I_i^{r_i}, \langle f \rangle_w)$ as the number of monic polynomial $g$ over $\Fq$  of  degree $m$ with $\langle g \rangle_w=\langle f \rangle_w$, where $g$ has $r_i$  distinct monic degree $i$ irreducible factors for each  $i\in T$.
In particular, if $m= \sum_{T} r_i$, then $N(m,\prod_{i\in T}I_i^{r_i}, \langle f \rangle_w)$  is the number of  monic polynomials over $\Fq$ of degree $m$  with both the first $w$ coefficients $f_1,\ldots,f_w$  and a factorization pattern in terms of their degrees are prescribed. \\
(2) Define $N^*
(m,\prod_{i\in T}I_i^{l_i}, \langle f \rangle_w)$ as the number of degree $m$ monic polynomial $g$ over $\Fq$ with $\langle g \rangle_w=\langle f \rangle_w$, where $g$ has $l_i$  monic degree $i$ irreducible factors counting multiplicity for each $i\in T$.
In particular, if $m= \sum_{T} l_i$, then $N^*(m,\prod_{i\in T}I_i^{l_i}, \langle f \rangle_w)$  is the number of monic polynomials over $\Fq$ of degree $m$  with both the first $w$ coefficients $f_1,\ldots,f_w$  and a factorization pattern in terms of their degrees  counting multiplicity are prescribed. \\
\end{definition}

Finding these numbers answers many previous known questions in the literature.  For example, the problem of counting monic irreducible  polynomials $f$ over $\Fq$ of degree $m$ with prescribed coefficients $f_1,\ldots,f_w$ is well studied.   Denote by  $I(m,\langle f \rangle_w)$  the number of degree $m$ monic irreducible polynomial $g$ over $\Fq$, where $\langle g \rangle_w=\langle f \rangle_w$. Then

\begin{align}
I(m,\langle f \rangle_w) = 
N^*(m,I_m^{1}\prod_{i=1}^{m-1}I_i^0, \langle f \rangle_w). 
\end{align}

For typographic convenience, we omit the subscript of $w$, when the value of $w$ is fixed.
If $w=0$, for any monic polynomial $f$, we have $\langle f \rangle\equiv 1 \pmod x=\langle  1 \rangle$. Thus, $I(m,\langle f \rangle)=|I_m|$, which is the total number of degree $m$ monic irreducible polynomials over $\Fq$.  This formula is known (see \cite{Finite Fields}) and is given by
\begin{align}
|I_m|=\frac{1}{m}\sum_{k|m}\mu(m/k)q^{k},
\end{align}
where $\mu$ is the M\"{o}bius function. 

The results for $w=1$ can  be found in  \cite{trace}. In this case, for each monic polynomial $f$, we have
$\langle f \rangle=1+\bet x \pmod{x^2}=\langle  x+\bet \rangle$ for some unique  $\bet\in\Fq$. For $m\ge 1$, $I(m,\langle  x+\bet \rangle)$ counts the number of monic irreducible polynomials of the form $x^m+\bet x^{m-1}+g(x)$ where $\bet\in\Fq$ is fixed and $g(x)\in\Fq[x]$ is a polynomial of degree at most $m-2$ which is allowed to vary. When $m$ is a multiple of $p$, the formula for $I(m,\langle  x+\bet \rangle)$ depends on whether $\bet=0$ or not.  Explicitly,  if $m=p^j n_0$ with $j>0$, then we have

\begin{align*} \label{YucasFormula}
I(m,\langle  x+\bet \rangle)=\frac{1}{mq}\sum_{k \mid n_0}\mu(k)q^{m/k}-\frac{v(\bet)}{mq}\sum_{ k \mid n_0}\mu(k)q^{m/kp},
\end{align*}
where $v(\bet)=q$ if $\bet = 0$ and $0$ otherwise.

If $p\nmid m$, then for all $\bet\in\Fq^*$, we have
\begin{align*}
I(m,\langle  x+\bet \rangle)=\frac{|I_m|}{q}=\frac{1}{mq}\sum_{k|m}\mu(m/k)q^k.
\end{align*}

We also note that there are formulas for $I(m,\langle f \rangle)$  for  $w=2$ (see \cite{MoisioRanto}), 
and also for $w=3$, when $q=2$ (see \cite{fityuc}). 

Another  special case of  our  general problem of computing $N$ and $N^*$ in Definition~\ref{genDef}  is  to determine the number of monic polynomials $f(x)$ with prescribed coefficients $f_1,\ldots,f_w$ and a given number of distinct  linear factors, or linear factors counting multiplicity. For $m\ge w$, $N(m,I_1^r,\langle f \rangle_w)$ counts the number of degree $m$ polynomials of the form $x^m+f_1x^{m-1}+\cdots+f_wx^{m-w}+g(x)$ with $r$ distinct linear factors, where $f_1,\ldots,f_w\in\Fq$ are fixed, and $g(x)\in\Fq[x]$ is a polynomial of degree at most $m-w-1$ that varies.

This number is important due to its applications in Reed-Solomon codes.
In general, if we write $\Fq=\{x_1,\ldots,x_q\}$, then a codeword in a Reed Solomon code of dimension $k$ and length $q$ is of the form $F=(f(x_1),\ldots,f(x_q))\in\Fq^q$, where $f(x)$ is a polynomial of degree at most $k-1$ over $\Fq$. In general, a vector $V\in\Fq^q$ can be written as $(v(x_1),\ldots,v(x_q))$ for some unique polynomial $v(x)$ of degree at most $q-1$, with the Lagrange Interpolation formula. The distance between $V$ and a codeword $F$ is the number of non-zero components in the vector $V-F$, which is equal to the  number of roots of the polynomial $v(x)-f(x)$.

An important problem in decoding messages in Reed-Solomon codes is to determine the number of codewords at a given distance of a received word.  Suppose in the above example that the word $V$ is received, and the polynomial $v(x)$ is a polynomial of degree $k+w$ for $w\ge 0$. Suppose without loss of generality that $v(x)$ is monic. Then the number of codewords of distance $q-r$ from the received word $V$ is the number of polynomials $f$ of degree at most $k-1$ where $v(x)-f(x)$ has at exactly $r$ distinct roots. Writing $v(x)=x^w+v_1x^{w-1}+\cdots+v_wx^k+c(x)$, for $v_1,\ldots,v_k\in\Fq$ and $c(x)\in\Fq[x]$ has degree at most $k-1$. The number of codewords at distance $q-r$ from $V$ is the number  of monic  polynomials $v(x)-f(x)$ where $f$ runs through all polynomials of degree at most $k-1$, and $v(x)-f(x)$ has $r$ distinct roots, which is equal to $N(k+w,I_1^r,\langle  v \rangle_w)$.





Both numbers
$N(m,I_1^r,\langle f \rangle_w)$  and $N^*(m,I_1^l,\langle f \rangle_w)$ 
 are studied in \cite{Knop} when $w=0$ in order to obtain the distribution of zeros of a random  monic polynomial of degree $m$,  with and without multiplicity counted. When $w=0$, for all monic polynomials $f$, we have $\langle f \rangle_w \equiv 1 \pmod x$, so dropping the subscript of $w$ and the modulo operation, we write $\langle f \rangle=1$ for all $f\in M$.  It is shown in \cite{Knop} that  the number    of degree $m$ monic polynomials over $\Fq$  with $l$ linear factors counting multiplicity is

\begin{align*}
N^*
(m,I_1^l, 1)=q^{m-l}{q+l-1\choose l}\sum_{j=0}^{m-l}{q\choose j}(-1)^{j}q^{-j}.
\end{align*}
If $m\ge q+l$, the formula simplifies to
\begin{align*}
N^*
(m,I_1^l, 1)=q^{m-l}{q+l-1\choose l}(1-1/q)^{q}.
\end{align*}

The number of degree $m$ monic polynomials over $\Fq$ with $r$ distinct linear factors is also given in \cite{Knop}. This number is
\begin{align*}
N
(m,I_1^r, 1)=q^{m-r}{q\choose r}\sum_{j=0}^{m-r}q^{-j}{q-r\choose j}(-1)^{j}.
\end{align*}
If $m\ge q$, this number becomes
\begin{align*}
N
(m,I_1^r, 1)=q^{m-r}{q\choose r}(1-1/q)^{q-r}.
\end{align*}

These results have been extended in recent years to allow for prescribing coefficients $f_1,\ldots,f_w$, where $w\ge 1$ is arbitrary, due to applications in Reed-Solomon codes  (see \cite{Zhou, Li}).  For the case $w=1$,  Zhou et al.  \cite{Zhou} studied the number of degree $m\ge 1$ polynomials over $\Fq$ with $r$ distinct roots of the form $x^m+\alp x^{m-1}+g(x)$, where $\alp\in\Fq$ is fixed, and $g(x)\in\Fq[x]$ is a varying polynomial of degree at most $m-2$. If $p\nmid m$, then the number is

\begin{align*}
N(m,I_1^r,\langle  x+\alp \rangle)=q^{m-r-1}{q\choose r}\sum_{j=0}^{m-r}q^{-j}{q-r\choose j}(-1)^{j}.
\end{align*}
If $p\mid m$, then the number is

\begin{align*}
N(m,I_1^r,\langle  x+\alp \rangle)=&q^{m-r-1}{q\choose r}\sum_{j=0}^{m-r}q^{-j}{q-r\choose j}(-1)^{j}\\
&+\frac{v(\alp)}{q}
{q/p \choose m/p}{m\choose r}(-1)^{m/p-r},
\end{align*}
where $v(\alp)=q\llbracket \alp=0 \rrbracket-1$.  Here we use the notation $\llbracket P \rrbracket$, which is equal to $1$ if $P$ is true and $\llbracket P \rrbracket = 0$ otherwise.

The exact and more complicated expressions are obtained in \cite{Zhou} for the number of monic degree $m\ge 2$ polynomials with $r$ distinct linear factors over $\Fq$ of the form $x^m+g(x)$ where $g(x)\in\Fq[x]$ is a varying polynomial of degree at most $m-3$. More recently, in \cite{Li}, an asymptotic bound on the number of degree $m\ge w$ polynomials with $r$ distinct linear factors over $\Fq$ of the form $x^{m}+b_1x^{m-1}+\cdots+b_wx^{m-w}+g(x)$, for fixed $b_1,\ldots,b_w\in\Fq$, and varied $g(x)\in\Fq[x]$ is obtained for $m\le q$.

In this paper, we derive a general expression for $N$ and $N^*$  from the generating functions over group-rings for any fixed $w\ge 0$ and any fixed $T\subset\mathbb{N}$.  Namely, we can  use this to find the number of monic polynomials over $\fqn$ of degree $m$ with their first $w$ coefficients prescribed and specific factorization pattern prescribed, counting multiplicity or not (see Theorems~\ref{Main1}-\ref{Main2}).  These general results for $N$ and $N^*$ both have many corollaries that improve known results.  In particular, we obtain some simpler consequences when the degree of polynomials is sufficiently large (see Theorems~\ref{Large1}-\ref{Large2}).  Then we focus on the cases when $w=0$ or $w=1$, and thus obtain some exact formulas  for $N$ and $N^*$ in those cases.  




We also apply our results to the study of smooth polynomials. 
Polynomials whose irreducible factors are all of degree at most $n$ are called $n$-smooth  and they have applications in security. 
The number of $n$-smooth polynomials of degree $m$ over $\Fq $ has already been considered  first by  Odlyzko \cite{Odlyzko}
who provided an asymptotic estimate when $m\rightarrow \infty$ for the case
$q = 2$ and $m^{1/100} < n < m^{99/100}$  using the saddle point method. This generalizes
to any prime power $q$; see \cite{Lovorn}. 
For $n$ large with respect to $m$, typically $n > c m  \log \log m/\log m$,  Car \cite{Car} has given an asymptotic expression for
this number in terms of the Dickman function. 
 Panario, Gourdon and Flajolet  \cite{Panarioetal} extended this range to  $n > (1 + \epsilon) (\log m)^{1/k}$, for a positive
integer constant $k$.  All these results are asymptotic results and  no exact formula is known previously.

When $w=0$,  we can simplify our general results and obtain two exact formulas for the number of  monic $n$-smooth polynomials of degree $m$ (see Corollaries~\ref{Smooth1},  \ref{Smooth2}).  
Similarly, when $w=1$,  we obtain two  exact formulas for the number of degree $m$ $n$-smooth polynomials of the form $x^m+\alp x^{m-1}+g(x)$ where $g$ is a polynomial of degree at most $m-2$ (see  Corollaries~\ref{Smooth3} and \ref{Smooth4}).
We also show that the corresponding different formulas are equivalent. 

The paper is organized as follows. In Section 2 we provide background definitions and preliminary results that were first introduced in   \cite{Paper, GKW21},   such as   generating functions  defined on the group algebra of the equivalent classes for polynomials over finite fields.   In Section 3 we demonstrate the generating function method over group rings  to count irreducible polynomials with prescribed coefficients, which were explored in  \cite{Paper, GKW21} earlier.   In Section 4 we develop the general results on counting  polynomials with prescribed coefficients and prescribed factorization pattern.  Then we further demonstrate our general methodology in different special cases such as large degree, $w=0$, and $w=1$ respectively.  These results can be found in Sections 5, 6, 7 respectively.





\section{Definitions and Notations}
In this section, we introduce the necessary background to be able to count monic polynomials with the first $w$ prescribed coefficients using the generating functions method.   A general combinatorial framework for counting irreducible polynomials with prescribed coefficients, using generating functions with coefficients from a group algebra, was developed in \cite{Paper} and  Section 2 of \cite{GKW21}. 

First, we fix $w\ge 0$. Recall that $M$ is set of monic polynomials  over $\Fq$.
For $f\in M$, we let $d(f)$ denote the degree of $f$,  write
$f=x^{d(f)}+f_1x^{d(f)-1}+\cdots+f_{d(f)}$, and set $f_j=0$ if $j>d(f)$. We also recall the notation

\begin{align*}
\langle f \rangle&=x^{d(f)}f(1/x)\pmod{x^{w+1}}\\
&=1+f_1x+\cdots+f_wx^w \pmod{x^{w+1}}.
\end{align*}

Fix $f\in M$ and  $d\geq w$. For $g \in M_d$,  $\langle g \rangle  = \langle f \rangle $ if and only if $g= x^d + f_1 x^{d-1} + \cdots + f_w x^{d-w} + c(x)$ for some $c(x) \in \fq[x]$ of degree at most $d-w -1$. Therefore there are $q^{d-w}$ monic polynomials
$g\in M_d$ satisfying $\langle g \rangle  = \langle f \rangle $. For convenience, we define
\begin{align}
G=\{\langle f \rangle:f\in M\} = \{ \langle x^w + f_1 x^{w-1} + \cdots + f_w \rangle : f_1, \ldots, f_w \in \fq\}.
\end{align}

From this definition, we have the following result.
\begin{prop} (Proposition ~1, \cite{GKW21}) \label{GroupProp}
$G$ is an abelian group  under multiplication $\langle f \rangle\langle g \rangle=\langle fg \rangle$ with identity $\langle  1 \rangle$.
\end{prop}

\begin{proof}
We first verify that $G$ is closed under the operation. For $f,g\in M$, 

\begin{align*}
\langle f \rangle\langle g \rangle&=x^{d(f)}f(1/x)x^{d(g)}g(1/x) \pmod{x^{w+1}}\\
&=x^{d(f)+d(g)}f(1/x) g(1/x) \pmod{x^{w+1}}\\
&=x^{d(fg)}(fg)(1/x) \pmod{x^{w+1}}\\
&=\langle fg \rangle.
\end{align*}

Using the fact that $M$ is a commutative monoid with identity $1$ and for $f,g\in M$, $\langle f \rangle\langle g \rangle=\langle fg \rangle$, we have that $G$ is a commutative monoid with identity $\langle  1 \rangle$.

For $f\in M$, 
$\langle f \rangle=1+f_1x+\cdots+f_wx^w\pmod{x^{w+1}}.$ Noting that $x\nmid 1+f_1x+\cdots+f_wx^w,$ we have
\begin{align*}
\gcd(x^{k+1},1+f_1x+\cdots+f_wx^w)=1.
\end{align*}
Hence, there exists $g_0+g_1x+\cdots+g_wx^w$ such that
\begin{align*}
(1+f_1x+\cdots+f_wx^w)(g_0+g_1x+\cdots+g_wx^w)\equiv 1\pmod{x^{w+1}}.
\end{align*}
It follows that $g_0=1*g_0=1$.
Let $g=x^w+g_1x^{w+1}+\cdots+g_w\in M$. We have
\begin{align*}
\langle f \rangle\langle g \rangle&=(1+f_1x+\cdots+f_wx^w)(1+g_1x+\cdots+g_wx^w)\pmod{x^{w+1}}\\
&=(1+f_1x+\cdots+f_wx^w)(g_0+g_1x+\cdots+gx^w)\pmod{x^{w+1}}\\
&=1\pmod{x^{w+1}}\\
&=\langle  1 \rangle.
\end{align*}
Hence, $\langle g \rangle$ is the multiplicative inverse of $\langle f \rangle$ in $G$.
\qed\end{proof}

Let $\C$ be the field of complex numbers.  In order to make use of the group $G$, it is convenient to define the following. 
\begin{definition}
Define $\C[G]$ to be the commutative ring of formal $\C$-linear combinations of elements of $G$. For convenience, write $0$ as the additive identity of $\C[G]$ and $1=\langle  1 \rangle$ as the multiplicative identity of $\C[G]$.  The elements of $\C[G]$ are of the form 
\begin{align*}
v=\sum_{\langle f \rangle\in G}v_{\langle f \rangle}{\langle f \rangle},
\end{align*}
where $v_{\langle f \rangle}\in\C.$
For $a=\sum_{\langle f \rangle\in G}a_{\langle f \rangle}{\langle f \rangle},b=\sum_{\langle f \rangle\in G}b_{\langle f \rangle}{\langle f \rangle}\in \C[G]$,  define  
\begin{align}
a+b&=\sum_{\langle f \rangle\in G}(a_{\langle f \rangle}+b_{\langle f \rangle}){\langle f \rangle},\\
ab&=\sum_{\langle f \rangle\in G}\sum_{\langle g \rangle\in G}a_{\langle g \rangle}b_{\langle f \rangle\langle g \rangle^-1}\langle f \rangle.
\end{align}
\end{definition}

To help with counting, the following elements in $\C [G]$ are useful.

\begin{definition}
Define 
\begin{align}
E&=\frac{1}{q^w}\sum_{\langle f \rangle\in G}{\langle f \rangle}\\
J&=1-E
\end{align}
\end{definition}

It is straightforward to verify that  $E$ and $J$ are orthogonal idempotents. 

\begin{prop}\label{BasicProp1} The following properties of $E$ and $J$ hold:\\
i) $E\langle g \rangle=E$ for any $\langle g \rangle\in G$.\\
ii) $E^2=E$.\\
iii) $EJ=0$.\\
iv) $J^2=J$.

\end{prop}

To connect to the polynomial counting problem, we recall that $M_d$ is the set of degree $d$ monic polynomials over $\Fq$, so $|M_d|=q^d$.
Suppose $f\in M$. Then for $d\ge w$, any $g\in M_d$ satisfies $\langle g \rangle=\langle f \rangle$ if and only if $g=x^d+f_1x^{d-1}+\cdots+f_wx^{d-w}+c(x)$ for some polynomial $c(x)$ of degree at most $d-w-1$ over $\Fq$. There are $q^{d-w}$ polyomials over $\Fq$ of degree at most $d-w-1$. Therefore, the number of polynomials $g\in M_d$ with $\langle g \rangle=\langle f \rangle$ is $q^{d-w}$. 
It follows that every $\langle f \rangle\in G$ is uniquely defined by the polynomial $h \in M_w$ that satisfies $\langle h \rangle=\langle f \rangle.$ We can therefore write 
\begin{align*}
G=\{\langle x^w+h_1x^{w-1}+\cdots+h_w \rangle: h_1,\ldots,h_w\in\Fq\}.
\end{align*}

Using the above defined notations and Proposition \ref{BasicProp1},  we can derive more facts that are useful when performing computations in $\C[G]$.

\begin{prop}\label{BasicProp2} The following properties hold:
\\
i) $E\sum_{f\in M_d}\langle f \rangle=q^dE$. \\
ii) $\sum_{f\in M_d}\langle f \rangle=q^dE$ for  $d\ge w$.\\
iii) $J\sum_{f\in M_d}\langle f \rangle=0$ for  $d\ge w$.\\
\end{prop}
\begin{proof}
i) For $f\in M_d$, $\langle f \rangle\in G$. From $|M_d|=q^d$ and Proposition \ref{BasicProp1}, we obtain
\[
E\sum_{f\in M_d}\langle f \rangle =\sum_{f\in M_d}E\langle f \rangle =  \sum_{f\in M_d}E =q^dE. 
\]

ii) Suppose $d\ge w$ and $h\in M_w$.  We note that there are $q^{d-w}$ polynomials $g\in M_d$ with $\langle g \rangle=\langle h \rangle$. It follows from the definition of $E$ that
\[
\sum_{f\in M_d}\langle f \rangle=q^{d-w}\sum_{h\in M_w}\langle h \rangle = \frac{q^d}{q^w} \sum_{\langle h \rangle\in G}\langle h \rangle =q^dE.
\]
iii) Suppose $d\ge w$. Then using $EJ=0$, we have
\[
J\sum_{f\in M_d}\langle f \rangle=q^dEJ=0.
\]
\qed
\end{proof}

A formal power series over the group ring $\C[G]$ is an important tool for counting polynomials.  As such, the following proposition is useful.

\begin{prop}\label{PowerProp1}
Suppose $A(z)$ is a formal power series over $\C[G]$.    If $K\in \C[G]$ satisfies $K^2=K$, then $KA(z)=KA(Kz)$.
In particular, we have  $EA(z)=EA(Ez)$,  $JA(z)=JA(Jz)$,  and $A(z)=EA(Ez)+JA(Jz)$.
\end{prop}

\begin{proof} Write $A(z)=\sum_{j\ge 0}a_jz^j$, where $a_j\in \C[G]$.  Since $K^2=K$, we have

\begin{align*}
KA(z)&=K\sum_{j\ge 0}a_jz^j\\
&=K\sum_{j\ge 0}Ka_jz^j\\
&=K\sum_{j\ge 0}Ka_jK^jz^j\\
&=K\sum_{j\ge 0}Ka_j(Kz)^j\\
&=K\sum_{j\ge 0}a_j(Kz)^j\\
&=KA(Kz).
\end{align*}
The rest of proof follows from Proposition~\ref{BasicProp1}. 
\qed\end{proof}

\section{Counting Irreducible Polynomials}
In this section, we demonstrate the generating functions method over group rings to recover some known results about the number of degree $m$ monic irreducible polynomials with the first few coefficients prescribed. In particular, we re-derive the total number of irreducible polynomials, and the number of irreducible polynomials of the form $x^m+\bet x^{m-1}+g(x)$ where $\bet\in\Fq$ is fixed, and $g(x)\in\Fq[x]$ of degree at most $m-2$ is varied. More details can be found in \cite{Paper, GKW21}, where this method was first introduced, and different cases such as  prescribed trace and norm,  or prescribed  multiple coefficients were considered respectively.

We recall that $I$ is the set of irreducible monic polynomials over $\Fq$. For $d\ge 1$, $I_d$ be the set of degree $d$ polynomials in $I$. For $f\in M$,  $I(d,\langle f \rangle)$ is the number of polynomials $g\in I_d$ with $\langle g \rangle=\langle f \rangle$. Define  the  generating function (GF) 
\begin{equation} 
F(z)=\sum_{f\in M}\langle f \rangle z^{d(f)} =1+\sum_{d\ge 1}\sum_{f\in M_d} \langle f \rangle z^d.\label{P1}
\end{equation} 

From the unique factorization of polynomials, we have

\begin{align} \nonumber
F(z)&=\prod_{f\in I}(1-\langle f \rangle z^{d(f)})^{-1}\\
&=\prod_{d\ge 1}\prod_{f\in I_d}(1-\langle f \rangle z^d)^{-1}\\ \nonumber
&=\prod_{d\ge 1}\prod_{\langle f \rangle\in G}(1-\langle f \rangle z^d)^{-I(d,\langle f \rangle)}.
\end{align} 

It follows that

\begin{align*}
\ln(F(z))&=\sum_{d\ge 1}\sum_{\langle f\rangle \in G}I(d,\langle f \rangle)\sum_{k\ge 1}\frac{\langle f \rangle^kz^{dk}}{k}\\
&=\sum_{m\ge 1}\sum_{d|m}\sum_{\langle f\rangle \in G}\frac{d}{m}I(d,\langle f \rangle)\langle f \rangle^{m/d}z^m.
\end{align*}
Let $N(m,\langle f \rangle)=m[\langle f \rangle z^m]\ln(F(z))$. Then 
\begin{align}
N(m,\langle f \rangle)&=\sum_{d|m}\sum_{\langle g \rangle\in G}dI(d,\langle g \rangle)\llbracket\langle g \rangle^{m/d}=\langle f \rangle\rrbracket.
\end{align}

\begin{prop} (Proposition~2, \cite{GKW21})  \label{IrredCount}
\begin{align*}
I(m,\langle f \rangle)&=\frac{1}{m}\sum_{k|m}\sum_{\langle g \rangle\in G}\mu(m/k)N(k,\langle g \rangle)\llbracket \langle g \rangle^{m/k}=\langle f \rangle \rrbracket.
\end{align*}
\end{prop}

\begin{proof} Using the fact that for $m\in\mathbb{N}$,
\begin{align*}
\sum_{d|m}\mu(d)=\llbracket m=1 \rrbracket, 
\end{align*}
we obtain
\begin{align*}
&\sum_{k|m}\sum_{\langle g \rangle\in G}\mu(m/k)N(k,\langle g \rangle)\llbracket \langle g \rangle^{m/k}=\langle f \rangle \rrbracket\\
&=\sum_{k|m}\sum_{\langle g \rangle\in G}\mu(m/k)\sum_{d|k}\sum_{\langle h \rangle \in G}dI(d,\langle h \rangle)\llbracket \langle h \rangle^{k/d}=\langle g \rangle \rrbracket \llbracket 
\langle g \rangle^{m/k}=\langle f \rangle \rrbracket\\
&=\sum_{k|m}\mu(m/k)\sum_{d|k}\sum_{\langle h \rangle\in G}dI(d,\langle h \rangle)\sum_{\langle g \rangle\in G}\llbracket \langle h \rangle^{k/d}=\langle g \rangle \rrbracket \llbracket 
\langle g \rangle^{m/k}=\langle f \rangle \rrbracket\\
&=\sum_{k|m}\mu(m/k)\sum_{d|k}\sum_{\langle h \rangle\in G}dI(d,\langle h \rangle)\llbracket 
\langle h \rangle^{m/d}=\langle f \rangle \rrbracket\\
&=\sum_{k|m}\mu(k)\sum_{d|\frac{m}{k}}\sum_{\langle h \rangle\in G}dI(d,\langle h \rangle)\llbracket 
\langle h \rangle^{m/d}=\langle f \rangle \rrbracket\\
&=\sum_{d|m}\sum_{\langle h \rangle\in G}dI(d,\langle h \rangle)\llbracket 
\langle h \rangle^{m/d}=\langle f \rangle  \rrbracket \sum_{k|\frac{m}{d}}\mu(k)\\
&=\sum_{d|m}\sum_{\langle h \rangle\in G}dI(d,\langle h \rangle)\llbracket 
\langle h \rangle^{m/d}=\langle f \rangle \rrbracket \llbracket d=m \rrbracket \\
&=m I(m,\langle f \rangle).
\end{align*}
Dividing by $m$, we obtain the result.\qed\end{proof}

\begin{prop} 
\begin{align*}
N(m,\langle f \rangle)=q^{m-w}+m[\langle f \rangle z^m]J\ln\left(1+\sum_{d=1}^{w-1}\sum_{f\in M_d}\langle f \rangle z^d\right).
\end{align*}
\end{prop}

\begin{proof}
Using (\ref{P1}),
\[
F(z)=1+\sum_{d\ge 1}\sum_{f\in M_d}\langle f \rangle z^d.
\]
Since $E$  and  $J$ are orthogonal idempotents, we have $1= E+J$, $E^d=E$,  and $J^d=J$.  Hence, from Propositions \ref{BasicProp2}  and \ref{PowerProp1}, it follows that 


\begin{align*}
\ln(F(z))&=E\ln(F(Ez))+J\ln(F(Jz))\\
&=E\ln\left(1+\sum_{d\ge 1}\sum_{f\in M_d}\langle f \rangle E^dz^d\right)+J\ln\left(1+\sum_{d\ge 1}\sum_{f\in M_d}\langle f \rangle J^dz^d\right)\\
&=E\ln\left(1+\sum_{d\ge 1}q^dE^dz^d\right)+J\ln\left(1+\sum_{d=1}^{w-1}\sum_{f\in M_d}\langle f \rangle J^dz^d\right)\\
&=E\ln\left(1+\sum_{d\ge 1}q^dz^d\right)+J\ln\left(1+\sum_{d=1}^{w-1}\sum_{f\in M_d}\langle f \rangle z^d\right)\\
&=E\ln\left(\frac{1}{1-qz}\right)+J\ln\left(1+\sum_{d=1}^{w-1}\sum_{f\in M_d}\langle f \rangle z^d\right)\\
&=E\sum_{k\ge 1}\frac{1}{k}q^kz^k+J\ln\left(1+\sum_{d=1}^{w-1}\sum_{f\in M_d}\langle f \rangle z^d\right).
\end{align*}
Using the definition $E=\frac{1}{q^w}\sum_{\langle f \rangle\in G}\langle f \rangle$ and extracting the coefficient $m[\langle f \rangle z^m]$ from $\ln(F(z))$, the result follows.
\qed\end{proof}

For $w=0,1$, we have  $N(m,\langle f \rangle)=q^{m-w}$. 
If $w=0$, then $G=\{1\}$. In this case, for $f\in M$, $\langle f \rangle=1$. It follows that

\begin{align}
|I_m|=I(m,1)=\frac{1}{m}\sum_{k|m}\mu(m/k)q^k.
\end{align}

If $w=1$, then
$G=\{1+\alp x \pmod {x^2}:\alp\in\Fq \}=\{\langle  x+\alp \rangle: \alp\in\Fq\}$. For $\alp,\bet\in\Fq$,  and $n\ge 1$,  we have  $\langle x+\alp \rangle^n =(1+\alp x)^n {\pmod {x^2}}=1+n \alp x  {\pmod {x^2}}\\
=<x+n\alp>$. Hence $\langle  x+\alp \rangle^n  =\langle  x+\bet \rangle$
if and only if $n\alp=\bet$.

Suppose $\bet\in\Fq$ and $n\ge 1$. Then

\begin{align}
\#\{\alp\in\Fq: n\alp=\bet\}=\llbracket p\nmid n \rrbracket+q\llbracket \bet =0 \rrbracket \llbracket p\mid n \rrbracket.
\end{align}

It follows from Proposition~\ref{IrredCount} that

\begin{align} \nonumber
I(m,\langle  x+\bet \rangle)&=\frac{1}{m}\sum_{k|m}\sum_{\alp\in\Fq}\mu(k)q^{m/k-1}\llbracket \langle  x+\alp \rangle^{k}=\langle  x+\bet \rangle\rrbracket\\ \nonumber
&=\frac{1}{mq}\sum_{k|m}\sum_{\alp\in\Fq}\mu(k)q^{m/k}\llbracket k\alp=\bet\rrbracket\\
&=\frac{1}{mq}\sum_{p\nmid k|m}\mu(k)q^{m/k}+\frac{ \llbracket \bet =0 \rrbracket   }{m}\sum_{p\mid k \mid m}\mu(k)q^{m/k}.
\end{align}
This is equivalent to  Equation~(\ref{YucasFormula}) due to Yucas \cite{trace}; see also \cite{Ruskey}.  In particular, if $p\nmid m$ and $w=1$, then for all $\bet\in\Fq$, 
\begin{align*}
I(m,\langle  x+\bet \rangle)=\frac{|I_m|}{q}=\frac{1}{mq}\sum_{k|m}\mu(k)q^{m/k}.
\end{align*}

\section{Factorization Problem: General Theory}

In this section, we develop the generating function method to find the number of monic polynomials over $\fqn$ of degree $m$ with their first $w$ coefficients prescribed and  the factorization pattern  in terms of  the degrees of irreducible factors prescribed. 
One can refer to \cite{DGP09, GP06, SF96} and references therein for related results on  general decomposable structures with prescribed patterns.

Let $T\subset\mathbb{N}$ be finite. 
For each $i\in T$ and $f\in M$, define $r_i(f)$ to be the number of distinct degree $i$ monic irreducible factors of $f$, and $l_i(f)$ to be the number of degree $i$  monic irreducible factors of $f$ counting multiplicity. Then

\begin{align}
r_i(f)=\sum_{g\in I_i}\llbracket g|f \rrbracket, \\
l_i(f)=\sum_{g\in I_i}\max\{k: g^k|f\}.
\end{align}

From Definition~\ref{genDef}, $N(m,\prod_{i\in T}I_i^{r_i}, \langle f \rangle)$ is the number of degree $m$ monic polynomials over $\Fq$ with $\langle g \rangle=\langle f \rangle$, where $g$ has $r_i$ distinct factors in $I_i$ for each  $i\in T$.
On the other hand, $N^*(m,\prod_{i\in T}I_i^{l_i}, \langle f \rangle)$ is the number of degree $m$ monic polynomials $g$ over $\Fq$ with $\langle g \rangle=\langle f \rangle$, where $g$ has $l_i$ factors in $I_i$ counting multiplicity for each $i\in T$.

For $i\in T$, let $u_i$ mark the irreducible monic polynomials of degree $i$. For $g\in I$, 
we define 
\[
u_g = \left\{ \begin{array}{ll}  u_i  & \mbox{if }  g\in I_i  \mbox{ for some } i \in T; \\  1 & \mbox{otherwise.} \end{array} \right.
\]

Define the $GFs$

\begin{align} 
G(z,u)&=\sum_{f\in M}\langle f \rangle z^{d(f)}\prod_{g\in I, g|f
}u_g =\sum_{f\in M}\langle f \rangle z^{d(f)}\prod_{i\in T}u_i^{r_i(f)},
\end{align}
\begin{align} 
H(z,u)&=\sum_{f\in M}\langle f \rangle z^{d(f)}\prod_{g\in I
}u_g^{\max\{k: g^k|f\}}=\sum_{f\in M}\langle f \rangle z^{d(f)}\prod_{i\in T}u_i^{l_i(f)}.
\end{align}

Note that

\begin{align}
[\langle f \rangle z^m\prod_{i \in T} u_i^{r_i}]G(z,u)&=N(m,\prod_{i\in T}I_i^{r_i},\langle f \rangle),
\end{align}
\begin{align}
[\langle f \rangle z^m\prod_{i\in T}u_i^{l_i}]H(z,u)&=N^*(m,\prod_{i\in T}I_i^{l_i},\langle f \rangle).
\end{align}

\begin{prop}  \label{expressionG}
The expression for $G(z,u)$ can be written as follows:
\begin{align*}
G(z,u)&=F(z)\prod_{i\in T}\prod_{g\in I_i}(1+\langle g  \rangle z ^{i}(u_i-1)),
\end{align*}
where $F(z) = \sum_{f\in M}\langle f \rangle z^{d(f)}$ is defined in (\ref{P1}). 
\end{prop}
\begin{proof}
From the fact that every monic polynomial factors uniquely into a product of monic irreducible polynomials,  we obtain
\begin{align*}
G(z,u)&=\sum_{f\in M}\langle f \rangle z^{d(f)}\prod_{g\in I, g|f
}u_g\\
&=\prod_{g\in I}\left(1+\sum_{k\ge 1}\langle g^k \rangle z^{d(g^k)}u_g\right)\\
&=\prod_{g\in I}\left(1+\sum_{k\ge 1}\langle g \rangle^k  z^{k(d(g))}u_g\right)\\
&=\prod_{g\in I}\left(1+\frac{u_g\langle g  \rangle z ^{d(g)}}{1-\langle g  \rangle z ^{d(g)}}\right)\\
&=\prod_{g\in I}\left(\frac{1-\langle g  \rangle z ^{d(g)}+u_g\langle g  \rangle z ^{d(g)}}{1-\langle g  \rangle z ^{d(g)}}\right)\\
&=\prod_{g\in I}\left(\frac{1+\langle g  \rangle z ^{d(g)}(u_g-1)}{1-\langle g  \rangle z ^{d(g)}}\right)\\
&=F(z)\prod_{g\in I}(1+\langle g  \rangle z ^{d(g)}(u_g-1)).
\end{align*}
Using the definition of $u_g$ we obtain the result.
\qed\end{proof}

We now derive a more explicit formula for $G(z,u)$. 

\begin{lemma}  Under the same notations as above,
\begin{align*} 
G(z,u)&=E\left(\sum_{d\ge 0}q^dz^d\right)\prod_{i\in T}\sum_{j_i=0}^{|I_i|}\sum_{r_i=0}^{j_i}{|I_i|\choose j_i}z^{ij_i}{j_i\choose r_i}u_i^{r_i}(-1)^{j_i-r_i}
\\&+J\left(\sum_{d=0}^{w-1}\sum_{f\in M_d}\langle f \rangle z^d\right)\prod_{i\in T}\prod_{g\in I_i}(1+\langle g  \rangle z ^{i}(u_i-1)).
\end{align*}
\end{lemma}
\begin{proof}
Using Proposition~\ref{expressionG}  and Equation~(\ref{P1}), we have
\begin{align*}
G(z,u)&=\left(\sum_{d\ge 0}\sum_{f\in M_d}\langle f \rangle z^d\right)\prod_{i\in T}\prod_{g\in I_i}(1+\langle g  \rangle z ^{i}(u_i-1)).
\end{align*}
Using $E^2=E$ and $E\langle f \rangle=E$, the follows that
\begin{align*}
EG(z,u)&=E\left(\sum_{d\ge 0}\sum_{f\in M_d}\langle f \rangle z^d\right)\prod_{i\in T}\prod_{g\in I_i}(1+\langle g  \rangle z ^{i}(u_i-1))\\
&=E\left(\sum_{d\ge 0}\sum_{f\in M_d}z^d\right)\prod_{i\in T}\prod_{g\in I_i}(1+z^{i}(u_i-1))\\
&=E\left(\sum_{d\ge 0}q^dz^d\right)\prod_{i\in T}(1+z^{i}(u_i-1))^{|I_i|}\\
&=E\left(\sum_{d\ge 0}q^dz^d\right)\prod_{i\in T}\sum_{j_i=0}^{|I_i|}{|I_i|\choose j_i}z^{ij_i}(u_i-1)^{j_i}\\
&=E\left(\sum_{d\ge 0}q^dz^d\right)\prod_{i\in T}\sum_{j_i=0}^{|I_i|}\sum_{r_i=0}^{j_i}{|I_i|\choose j_i}z^{ij_i}{j_i\choose r_i}u_i^{r_i}(-1)^{j_i-r_i}. 
\end{align*}

Using $J\sum_{f\in M_d}\langle f \rangle=0$ for $d\ge w$, we have
\begin{align*}
JG(z,u)&=J\left(\sum_{d\ge 0}\sum_{f\in M_d}\langle f \rangle z^d\right)\prod_{i\in T}\prod_{g\in I_i}(1+\langle g  \rangle z ^{i}(u_i-1))\\
&=J\left(\sum_{d=0}^{w-1}\sum_{f\in M_d}\langle f \rangle z^d\right)\prod_{i\in T}\prod_{g\in I_i}(1+\langle g  \rangle z ^{i}(u_i-1)).
\end{align*}
Using $G(z,u)=EG(z,u)+JG(z,u)$, we obtain the result.
\qed\end{proof}

\begin{thm}\label{Main1}
Let $T\subset\mathbb{N}$ be finite and $I_i$ be the set of monic irreducible polynomials of degree $i$.  Let $f$ be a fixed monic polynomial over $\fq$ with degree  $d$ and $w$ be a fixed positive integer. 
The number of degree $m$ monic polynomials $g$ over $\Fq$  with the first $w$ coefficients prescribed as those of $f$ and 
 $g$ has $r_i$  distinct  factors in $I_i$ for each $i\in T$ is 
\begin{align*}
&N(m,\prod_{i\in T}I_i^{r_i},\langle f \rangle)\\&=q^{m-w}\prod_{i\in T}{|I_i|\choose r_i}q^{-ir_i}\sum_{j_i=0}^{|I_i|-r_i}q^{-ij_i}{|I_i|-r_i\choose j_i}(-1)^{j_i}\llbracket \sum_{i\in T}{i(r_i+j_i)\le m} \rrbracket\\
&+[\langle f \rangle z^m\prod_{i\in T}u_i^{r_i}]J\left(\sum_{d=0}^{w-1}\sum_{f\in M_d}\langle f \rangle z^d\right)\prod_{i\in T}\prod_{g\in I_i}(1+\langle g  \rangle z ^{i}(u_i-1)).
\end{align*}
\end{thm}

\begin{proof} First
we note that
\begin{align*}
N(m,\prod_{i\in T}I_i^{r_i},\langle f \rangle)&=[\langle f \rangle z^m\prod_{i\in T}u_i^{r_i}]G(z,u),
\end{align*}
and 
\begin{align*} 
G(z,u)&=E\left(\sum_{d\ge 0}q^dz^d\right)\prod_{i\in T}\sum_{j_i=0}^{|I_i|}\sum_{r_i=0}^{j_i}{|I_i|\choose j_i}z^{ij_i}{j_i\choose r_i}u_i^{r_i}(-1)^{j_i-r_i}
\\&+J\left(\sum_{d=0}^{w-1}\sum_{f\in M_d}\langle f \rangle z^d\right)\prod_{i\in T}\prod_{g\in I_i}(1+\langle g  \rangle z ^{i}(u_i-1)).
\end{align*}

Note that for $0\le r_i\le j_i\le |I_i|$, 
\begin{align*}
{|I_i|\choose j_i}{j_i\choose r_i}&=\frac{|I_i|!}{j_i!(|I_i|-j_i)!}\frac{j_i!}{r_i!(j_i-r_i)!}\\
&=\frac{|I_i|!}{r_i!(|I_i|-j_i)!(j_i-r_i)!}\\
&=\frac{|I_i|!}{r_i!(|I_i|-r_i)!}\frac{(|I_i|-r_i)!}{(j_i-r_i)!(|I_i|-
j_i)!}\\
&={|I_i|\choose r_i}{|I_i|-r_i\choose j_i-r_i}.
\end{align*}
It follows that
\begin{align*}
EG(z,u)&=E\left(\sum_{d\ge 0}q^dz^d\right)\prod_{i\in T}\sum_{j_i=0}^{|I_i|}\sum_{r_i=0}^{j_i}{|I_i|\choose r_i}z^{ij_i}{|I_i|-r_i\choose j_i-r_i}(-1)^{j_i-r_i}{u_i^{r_i}}\\
&=E\left(\sum_{d\ge 0}q^dz^d\right)\prod_{i\in T}\sum_{r_i=0}^{|I_i|}{|I_i|\choose r_i}{u_i^{r_i}\sum_{j_i=r_i}^{|I_i|}z^{ij_i}{|I_i|-r_i\choose j_i-r_i}}(-1)^{j_i-r_i}\\
&=E\left(\sum_{d\ge 0}q^dz^d\right)\prod_{i\in T}\sum_{r_i=0}^{|I_i|}{|I_i|\choose r_i}{u_i^{r_i}}z^{ir_i}\sum_{j_i=0}^{|I_i|-r_i}z^{ij_i}{|I_i|-r_i\choose j_i}(-1)^{j_i}.
\end{align*}

Extracting the coefficient of $\prod_{i\in T}u_i^{r_i}$, we have

\begin{align*}
[\prod_{i\in T}u_i^{r_i}]EG(z,u)&=E\left(\sum_{d\ge 0}q^dz^d\right)\prod_{i\in T}{|I_i|\choose r_i}z^{ir_i}\sum_{j_i=0}^{|I_i|-r_i}z^{ij_i}{|I_i|-r_i\choose j_i}(-1)^{j_i}.
\end{align*}

Extracting $z^m$, we have

\begin{align*}
&[z^m\prod_{i\in T}u_i^{r_i}]EG(z,u)\\
&=Eq^m\prod_{i\in T}{|I_i|\choose r_i}q^{-ir_i}\sum_{j_i=0}^{|I_i|-r_i}q^{-ij_i}{|I_i|-r_i\choose j_i}(-1)^{j_i}\llbracket \sum_{i\in T}{i(r_i+j_i)\le m} \rrbracket.
\end{align*}

Hence, using the definition of $E$, extracting the coefficient of $\langle f \rangle$, we obtain

\begin{align*}
&[\langle f \rangle z^m\prod_{i\in T}u_i^{r_i}]EG(z,u)\\
&=q^{m-w}\prod_{i\in T}{|I_i|\choose r_i}q^{-ir_i}\sum_{j_i=0}^{|I_i|-r_i}q^{-ij_i}{|I_i|-r_i\choose j_i}(-1)^{j_i}\llbracket \sum_{i\in T}{i(r_i+j_i)\le m} \rrbracket.
\end{align*}
Adding $[\langle f \rangle z^m\prod_{i\in T}u_i^{r_i}]JG(z,u)$, we
obtain the result.
\qed\end{proof}

Similarly, we obtain the following. 

\begin{prop} \label{expressionH}
The expression for $H(z,u)$ can be written as follows:
\begin{align*}
H(z,u)&=F(z)\prod_{i\in T}\prod_{g\in I_i}\left(
\frac{1-\langle g  \rangle z ^{i}}{1-\langle g  \rangle z ^{i}u_i}\right).
\end{align*}
\end{prop}
\begin{proof}
\begin{align*}
H(z,u)&=\sum_{f\in M}\langle f \rangle z^{d(f)}\prod_{g\in I
}u_g^{\max\{k: g^k|f\}}\\
&=\prod_{g\in I}\left(\sum_{k\ge 0}\langle g^k \rangle  u_g^k z^{d{(g^k)}}\right)\\
&=\prod_{g\in I}\left(\sum_{k\ge 0}\langle g \rangle^k  u_g^k z^{kd(g)}\right)\\
&=\prod_{g\in I}\left(
\frac{1}{1-\langle g  \rangle z ^{d(g)}u_g}\right)\\
&=F(z)\prod_{g\in I}\left( 
\frac{1-\langle g  \rangle z ^{d(g)}}{1-\langle g  \rangle z ^{d(g)}u_g}\right).
\end{align*}
Using the definition of $u_g$, we obtain the result.
\qed\end{proof}

\begin{lemma} We have the following formula for $H(z,u)$:

\begin{align*}
H(z,u)&=E\left(\sum_{d\ge 0}q^dz^d\right)
\prod_{i\in T}\sum_{j_i=0}^{|I_i|}{|I_i|\choose j_i}(-1)^{j_i}z^{ij_i}\sum_{r_i\ge 0}{|I_i|+r_i-1\choose r_i}z^{ir_i}u_i^{r_i}\\
&+J\left(\sum_{d=0}^{w-1}\sum_{f\in M_d}\langle f \rangle z^d\right)\prod_{i\in T}\prod_{g\in I_i}\left(
\frac{1-\langle g  \rangle z ^{i}}{1-\langle g  \rangle z ^{i}u_i}\right).
\end{align*}
\end{lemma}
\begin{proof}
Using  Proposition~\ref{expressionH} and Equation~(\ref{P1}), we have
\begin{align*}
H(z,u)
&=\left(\sum_{d\ge 0}\sum_{f\in M_d}\langle f \rangle z^d\right)
\prod_{i\in T}\prod_{g\in I_i}\left(
\frac{1-\langle g  \rangle z ^{i}}{1-\langle g  \rangle z ^{i}u_i}\right).
\end{align*}
It follows that
\begin{align*}
EH(z,u)&=E\left(\sum_{d\ge 0}\sum_{f\in M_d}\langle f \rangle z^d\right)
\prod_{i\in T}\prod_{g\in I_i}\left(
\frac{1-\langle g  \rangle z ^{i}}{1-\langle g  \rangle z ^{i}u_i}\right)\\
&=E\left(\sum_{d\ge 0}\sum_{f\in M_d}z^d\right)
\prod_{i\in T}\prod_{g\in I_i}\left(
\frac{1-z^{i}}{1-z^{i}u_i}\right)\\
&=E\left(\sum_{d\ge 0}q^dz^d\right)
\prod_{i\in T}\left(
\frac{1-z^{i}}{1-z^{i}u_i}\right)^{|I_i|}\\
&=E\left(\sum_{d\ge 0}q^dz^d\right)
\prod_{i\in T}\sum_{j_i=0}^{|I_i|}{|I_i|\choose j_i}(-1)^{j_i}z^{ij_i}\sum_{r_i\ge 0}{|I_i|+r_i-1\choose r_i}z^{ir_i}u_i^{r_i}.
\end{align*}
Using $J\sum_{f\in M_d}\langle f \rangle=0$ for $d\ge w$, we have
\begin{align*}
JH(z,u)&=J\left(\sum_{d\ge 0}\sum_{f\in M_d}\langle f \rangle z^d\right)\prod_{i\in T}\prod_{g\in I_i}\left(
\frac{1-\langle g  \rangle z ^{i}}{1-\langle g  \rangle z ^{i}u_i}\right)\\
&=J\left(\sum_{d=0}^{w-1}\sum_{f\in M_d}\langle f \rangle z^d\right)\prod_{i\in T}\prod_{g\in I_i}\left(
\frac{1-\langle g  \rangle z ^{i}}{1-\langle g  \rangle z ^{i}u_i}\right).
\end{align*}
Using $H(z,u)=EH(z,u)+JH(z,u)$, the result follows.
\qed\end{proof}

\begin{thm}\label{Main2} 
Let $T\subset\mathbb{N}$ be finite and $I_i$ be the set of monic irreducible polynomials of degree $i$.  Let $f$ be a fixed  monic polynomial over $\fq$ with degree  $d$ and $w$ be a fixed positive integer. 
The number of degree $m$ monic polynomials $g$ over $\Fq$  with the first $w$ coefficients prescribed as those of $f$ and 
 $g$ has $r_i$   factors in $I_i$  counting multiplicity  for each $i\in T$ is 
\begin{align*}
&N^*(m,\prod_{i\in T}I_i^{l_i},\langle f \rangle)\\&=q^{m-w}\prod_{i\in T}{|I_i|+l_i-1\choose l_i}q^{-il_i}\sum_{j_i=0}^{|I_i|}{|I_i|\choose j_i}(-1)^{j_i}q^{-ij_i}\llbracket \sum_{i\in T} i(l_i+j_i) \le m\rrbracket\\
&+[\langle f \rangle z^m\prod_{i\in T}u_i^{l_i}]J\left(\sum_{d=0}^{w-1}\sum_{f\in M_d}\langle f \rangle z^d\right)\prod_{i\in T}\prod_{g\in I_i}\left(
\frac{1-\langle g  \rangle z ^{i}}{1-\langle g  \rangle z ^{i}u_i}\right).
\end{align*}
\end{thm}
\begin{proof}
We note  that 
\begin{align*}
N^*(m,\prod_{i\in T}I_i^{l_i},\langle f \rangle)=[\langle f \rangle z^m\prod_{i\in T}u_i^{l_i}]H(z,u),
\end{align*}
and 
\begin{align*}
H(z,u)&=E\left(\sum_{d\ge 0}q^dz^d\right)
\prod_{i\in T}\sum_{j_i=0}^{|I_i|}{|I_i|\choose j_i}(-1)^{j_i}z^{ij_i}\sum_{r_i\ge 0}{|I_i|+r_i-1\choose r_i}z^{ir_i}u_i^{r_i}\\
&+J\left(\sum_{d=0}^{w-1}\sum_{f\in M_d}\langle f \rangle z^d\right)\prod_{i\in T}\prod_{g\in I_i}\left(
\frac{1-\langle g  \rangle z ^{i}}{1-\langle g  \rangle z ^{i}u_i}\right).
\end{align*}
For the first line, extracting the coefficient of $\prod_{i\in T}u_i^{l_i},$ we have
\begin{align*}
[\prod_{i\in T}u_i^{l_i}]EH(z,u)=E\left(\sum_{d\ge 0}q^dz^d\right)
\prod_{i\in T}{|I_i|+l_i-1\choose l_i}z^{il_i}\sum_{j_i=0}^{|I_i|}{|I_i|\choose j_i}(-1)^{j_i}z^{ij_i}.
\end{align*}
Extracting the coefficient of $z^m$, we have
\begin{align*}
&[z^m\prod_{i\in T}u_i^{l_i}]EH(z,u)\\&=Eq^{m}\prod_{i\in T}{|I_i|+l_i-1\choose l_i}q^{-il_i}\sum_{j_i=0}^{|I_i|}{|I_i|\choose j_i}(-1)^{j_i}q^{-ij_i}\llbracket \sum_{i\in T} i(l_i+j_i) \le m\rrbracket.\\
\end{align*}
Using the definition of $E$ and extracting $\langle f \rangle$, we have
\begin{align*}
&[\langle f \rangle z^m\prod_{i\in T}u_i^{l_i}]EH(z,u)\\&=q^{m-w}\prod_{i\in T}{|I_i|+l_i-1\choose l_i}q^{-il_i}\sum_{j_i=0}^{|I_i|}{|I_i|\choose j_i}(-1)^{j_i}q^{-ij_i}\llbracket \sum_{i\in T} i(l_i+j_i) \le m\rrbracket.\\
\end{align*}
Adding $[\langle f \rangle z^m\prod_{i\in T}u_i^{l_i}]JH(z,u)$, we obtain the result.
\qed\end{proof}

\section{Large Degree Polynomials}

In this section, we derive some simpler consequences under certain restrictions such that 
the degree $m$ of the desired polynomials is very large, comparing to the factorization pattern and the number of prescribed
coefficients.

\begin{thm}\label{Large1} Suppose that $\sum_{i\in T}{i|I_i|}\le m-w$. Then
\begin{align*}
N(m,\prod_{i\in T}I_i^{r_i},\langle f \rangle)&=q^{m-w}\prod_{i\in T}{|I_i|\choose r_i}(1/q^{i})^{r_i}(1-1/{q^{i}})^{|I_i|-r_i}.
\end{align*}
\end{thm}
\begin{proof}
By Theorem~\ref{Main1},  we have 
\begin{align*}
&N(m,\prod_{i\in T}I_i^{r_i},\langle f \rangle)\\&=q^{m-w}\prod_{i\in T}{|I_i|\choose r_i}q^{-ir_i}\sum_{j_i=0}^{|I_i|-r_i}q^{-ij_i}{|I_i|-r_i\choose j_i}(-1)^{j_i}\llbracket \sum_{i\in T}{i(r_i+j_i)\le m} \rrbracket\\
&+[\langle f \rangle z^m\prod_{i\in T}u_i^{r_i}]J\left(\sum_{d=0}^{w-1}\sum_{f\in M_d}\langle f \rangle z^d\right)\prod_{i\in T}\prod_{g\in I_i}
(1+\langle g  \rangle z ^{i}(u_i-1)).
\end{align*}
If $j_i\le |I_i|-r_i$, then $r_i+j_i\le |I_i|$ and thus $\sum_{i\in T}{i(r_i+j_i)\le \sum_{i\in T}{i|I_i|}\le  m}$.  Hence the bracket condition for the first line holds.   The term on the second line is a polynomial in $z$ of degree less than $w+\sum_{i\in T}i|I_i|\le m$. Thus
\begin{align*}
[\langle f \rangle z^m\prod_{i\in T}u_i^{r_i}]J\left(\sum_{d=0}^{w-1}\sum_{f\in M_d}\langle f \rangle z^d\right)\prod_{i\in T}\prod_{g\in I_i}
(1+\langle g  \rangle z ^{i}(u_i-1))=0.
\end{align*}

Therefore,
\begin{align*}
N(m,\prod_{i\in T}I_i^{r_i},\langle f \rangle)&=q^{m-w}\prod_{i\in T}{|I_i|\choose r_i}q^{-ir_i}\sum_{j_i=0}^{|I_i|-r_i}q^{-ij_i}{|I_i|-r_i\choose j_i}(-1)^{j_i}\\
&=q^{m-w}\prod_{i\in T}{|I_i|\choose r_i}(1/q^{i})^{r_i}\sum_{j_i=0}^{|I_i|-r_i}(-1/q^{i})^{j_i}{|I_i|-r_i\choose j_i}\\
&=q^{m-w}\prod_{i\in T}{|I_i|\choose r_i}(1/q^{i})^{r_i}(1-1/q^{i})^{|I_i|-r_i}.\qed
\end{align*}
\end{proof}
When we fix $T$ to contain only one single degree $i$, the formula for $N$ further simplifies.
\begin{cor}\label{LargeCor1} Fix $i\ge 1$. Suppose that $m\ge i|I_i|+w$. Then the number of polynomials  $x^m+a_1x^{m-1}+\cdots+a_wx^{m-w}+g(x)$ with $g(x)\in\Fq[x]$ of degree at most $m-w-1$, that have $r$ distinct irreducible factors of degree $i$, is 
\begin{align*}
&N(m,{I_i}^{r}, \langle x^w+a_1x^{w-1}+\cdots+a_w \rangle)=q^{m-w}{|I_i|\choose r}(1/q^{i})^{r}(1-1/{q^{i}})^{|I_i|-r}.
\end{align*}
\end{cor}
Setting $i=1$, we obtain the following results about the number of monic polynomials with a given number of distinct linear factors with the highest few consecutive terms prescribed.
\begin{cor}
Suppose that $m\ge q+w$. Fix $a_1,\ldots,a_w\in\Fq$. Then the number of polynomials  $x^m+a_1x^{m-1}+\cdots+a_wx^{m-w}+g(x)$ with $g(x)\in\Fq[x]$ of degree at most $m-w-1$, that have $r$ distinct linear factors, is
\begin{align*}
&N(m,{I_1}^{r}, \langle x^w+a_1x^{w-1}+\cdots+a_w \rangle)=q^{m-w-r}{q\choose r}(1-1/q)^{q-r}.
\end{align*}
\end{cor}
Furthermore, setting $w=0$, we obtain the following known result.

\begin{cor} [Theorem 3 \cite{Knop}]
Suppose that $m\ge q$.  Then the number of monic polynomials of degree $m$ that have $r$ distinct linear factors is
\begin{align*}
&N(m,{I_1}^{r},1)=q^{m-w-r}{q\choose r}(1-1/q)^{q-r}.
\end{align*}
\end{cor}

Similarly, we obtain the analogous result when considering possible repeated factors. 
\begin{thm}\label{Large2}
Suppose that $\sum_{i\in T}{i(|I_i|+l_i)}\le m-w$. Then
\begin{align*}
N^*(m,\prod_{i\in T}I_i^{l_i},\langle f \rangle)&=q^{m-w}\prod_{i\in T}{|I_i|+l_i-1\choose l_i}(1/q^{i})^{l_i}(1-1/{q^{i}})^{|I_i|}.
\end{align*}
\end{thm}
\begin{proof}
From Theorem~\ref{Main2}, we obtain
\begin{align*}
&N^*(m,\prod_{i\in T}I_i^{l_i},\langle f \rangle)\\&=q^{m-w}\prod_{i\in T}{|I_i|+l_i-1\choose l_i}q^{-il_i}\sum_{j_i=0}^{|I_i|}{|I_i|\choose j_i}(-1)^{j_i}q^{-ij_i}\llbracket \sum_{i\in T} i(l_i+j_i) \le m\rrbracket\\
&+[\langle f \rangle z^m\prod_{i\in T}u_i^{l_i}]J\left(\sum_{d=0}^{w-1}\sum_{f\in M_d}\langle f \rangle z^d\right)\prod_{i\in T}\prod_{g\in I_i}\left(
\frac{1-\langle g  \rangle z ^{i}}{1-\langle g  \rangle z ^{i}u_i}\right).
\end{align*}
For the first line, any term $j_i$ in the sum satisfies $j_i\le |I_i|$, so the bracket condition holds. For the second line,
\begin{align*}
\prod_{g\in I_i}\frac{1}{1-\langle g  \rangle z ^{i}u_i}=\sum_{j\ge 0}a_{ij}z^{ji}{u_i^j}
\end{align*}
with $a_{ij}\in\C[G]$.
This means that
\begin{align*}
[\prod_{i\in T}u_i^{l_i}]J\left(\sum_{d=0}^{w-1}\sum_{f\in M_d}\langle f \rangle z^d\right)\prod_{i\in T}\prod_{g\in I_i}\left(
\frac{1-\langle g  \rangle z ^{i}}{1-\langle g  \rangle z ^{i}u_i}\right)
\end{align*}
is a polynomial in $z$ over $\C[G]$ of degree less than $\sum_{i=i}^ni(|I_i|+l_i)+w\le m$.
Hence,
\begin{align*}
[\langle f \rangle z^m\prod_{i\in T}u_i^{l_i}]J\left(\sum_{d=0}^{w-1}\sum_{f\in M_d}\langle f \rangle z^d\right)\prod_{i\in T}\prod_{g\in I_i}\left(
\frac{1-\langle g  \rangle z ^{i}}{1-\langle g  \rangle z ^{i}u_i}\right)=0.
\end{align*}
Therefore,
\begin{align*}
N^*(m,\prod_{i\in T}I_i^{l_i},\langle f \rangle)&=q^{m-w}\prod_{i\in T}{|I_i|+l_i-1\choose l_i}q^{-il_i}\sum_{j_i=0}^{|I_i|}{|I_i|\choose j_i}(-1)^{j_i}q^{-ij_i}\\
&=q^{m-w}\prod_{i\in T}{|I_i|+l_i-1\choose l_i}(1/q^{i})^{l_i}\sum_{j_i=0}^{|I_i|}{|I_i|\choose j_i}(-1/q^{i})^{j_i}\\
&=q^{m-w}\prod_{i\in T}{|I_i|+l_i-1\choose l_i}(1/q^{i})^{l_i}(1-1/q^{i})^{|I_i|}.\qed
\end{align*}
\end{proof}

Again, when we fix $T$ to contain only one single degree $i$, the formula for $N^*$ further simplifies.
\begin{cor}\label{LargeCor2}
Suppose that $m\ge{i(|I_i|+l)+w}$. Then the number of polynomials  $x^m+a_1x^{m-1}+\cdots+a_wx^{m-w}+g(x)$ with $g(x)\in\Fq[x]$ of degree at most $m-w-1$, that have $l$ degree $i$ irreducible factors counting multiplicity, is
\begin{align*}
N^*(m,I_i^{l},\langle f \rangle)&=q^{m-w}{|I_i|+l-1\choose l}(1/q^{i})^{l}(1-1/{q^{i}})^{|I_i|}.
\end{align*}
\end{cor}
Setting $i=1$, we obtain the following results about the number of monic polynomials with a given number of linear factors counting multiplicity with the highest few consecutive terms prescribed.
\begin{cor}
Suppose that $m\ge q+l+w$. Fix $a_1,\ldots,a_w\in\Fq$. Then the number of polynomials  $x^m+a_1x^{m-1}+\cdots+a_wx^{m-w}+g(x)$ with $g(x)\in\Fq[x]$ of degree at most $m-w-1$, that have $l$ linear factors counting multiplicity, is
\begin{align*}
&N^*(m,{I_1}^{l}, \langle x^w+a_1x^{w-1}+\cdots+a_w \rangle )=q^{m-w-l}{q+l-1\choose l}(1-1/q)^{q}.
\end{align*}
\end{cor}

Setting $w=0$, we obtain the following known result.
\begin{cor}[Theorem 1 \cite{Knop}]
Suppose that $m\ge q+l$. Then the number of monic degree $m$ polynomials that have $l$ linear factors counting multiplicity, is
\begin{align*}
&N^*(m,{I_1}^{l},1)=q^{m-l}{q+l-1\choose l}(1-1/q)^{q}.
\end{align*}
\end{cor}

\section{$w=0$: no prescribed coefficients}
In this section, we focus on the special case when $w=0$.  In this case, no coefficients are prescribed, so we are simply counting monic polynomials that has certain factorization pattern.




In fact, when $w=0$, we have $G=\{1\}$. This means that for all $\langle f \rangle\in G, \langle f \rangle=1$.  Hence we have the following consequence  from Theorem \ref{Main1}. 

\begin{cor}\label{Gen1} 
We have the following formula for $N(m,\prod_{i\in T}I_i^{r_i},\langle f \rangle)$ when $w=0$:
\begin{align*}
&N(m,\prod_{i\in T}I_i^{r_i},  1  )\\&=q^{m}\prod_{i\in T}{|I_i|\choose r_i}q^{-ir_i}\sum_{j_i=0}^{|I_i|-r_i}q^{-ij_i}{|I_i|-r_i\choose j_i}(-1)^{j_i}\llbracket \sum_{i\in T}{i(r_i+j_i)\le m} \rrbracket.\\
\end{align*}
\end{cor}

When we set $T=\{i\}$, we obtain an expression for the number of degree $m$ monic polynomials with a given number of distinct degree $i$ irreducible factors.

\begin{cor}[Theorem 3 \cite{Knop2}]\label{GenCor1} The number of degree $m$ monic polynomials over $\Fq$ that contains $r$ distinct irreducible  factors of degree $i$ is
\begin{align*}
N(m,{I_i}^r,1)&=q^{m-ir}{|I_i|\choose r}\sum_{j=0}^{\lfloor m/i \rfloor-r}q^{-j}{|I_i|-r\choose j}(-1)^{j}.
\end{align*}
\end{cor}

The case  when  $i=1$ was known in \cite{Knop}. 

\begin{cor}[Theorem 3 \cite{Knop}] The number of degree $m$ monic polynomials over $\Fq$ with $r$ distinct linear factors is
\begin{align*}
N(m,{I_1}^r,1)&=q^{m-r}{q\choose r}\sum_{j=0}^{m-r}q^{-j}{q-r\choose j}(-1)^{j}.
\end{align*}
\end{cor}

In addition to these results, we can also obtain an exact formula for the number of monic $n$-smooth polynomials of degree $m$  by using  Corollary \ref{Gen1}. This formula is useful when $n$ is close in size to $m$.

\begin{cor}\label{Smooth1} The number of monic $n$-smooth polynomials  over $\Fq$ with degree $m$  is
\begin{align*}
N(m,\prod_{i=n+1}^mI_i^{0},1)=q^{m}\prod_{i=n+1}^m\sum_{j_i=0}^{|I_i|}q^{-ij_i}{|I_i|\choose j_i}(-1)^{j_i}\llbracket \sum_{i=n+1}^m{ij_i\le m} \rrbracket.
\end{align*}
\end{cor}
\begin{proof}
A degree $m$ polynomial is $n$-smooth if it contains no factors above degree $n$. Hence, we can obtain  this result by setting $T=\{n+1,\ldots,m\}$ and then checking polynomials with no irreducible factors in $T$ with  Corollary \ref{Gen1}.\qed
\end{proof}

We can obtain a similar result to Corollary \ref{Gen1} when multiplicity of the factors are counted.  Indeed, the following result follows from Theorem \ref{Main2}.  

\begin{cor}\label{Gen2} 
Let $w=0$. Then
\begin{align*}
&N^*(m,\prod_{i\in T}I_i^{l_i}, 1 )\\&=q^{m}\prod_{i\in T}{|I_i|+l_i-1\choose l_i}q^{-il_i}\sum_{j_i=0}^{|I_i|}{|I_i|\choose j_i}(-1)^{j_i}q^{-ij_i}\llbracket \sum_{i\in T} i(l_i+j_i) \le m\rrbracket.\\
\end{align*}
\end{cor}


When we set $T=\{i\}$, we obtain an expression for the number of degree $m$ monic polynomials with a given number of degree $i$ irreducible factors counting multiplicity.

\begin{cor}[Theorem 1 \cite{Knop2}]\label{GenCor2} The number of degree $m$ monic polynomials over $\Fq$ with $l$ irreducible degree $i$ factors counting multiplicity is
\begin{align*}
N^*(m,I_i^{l},1)&=q^{m-il}{|I_i|+l-1\choose l}\sum_{j=0}^{\lfloor m/i \rfloor-l}{|I_i|\choose j}(-1)^{j}q^{-j}.
\end{align*}
\end{cor}

The case when $i=1$ was known in \cite{Knop}. 

\begin{cor}[Theorem 1 \cite{Knop}] The number of degree $m$ monic polynomials over $\Fq$ with $l$ linear factors counting multiplicity is
\begin{align*}
N^*(m,I_1^{l},1)&=q^{m-l}{q+l-1\choose l}\sum_{j=0}^{m-l}{q\choose j}(-1)^{j}q^{-j}.
\end{align*}
\end{cor}

In addition to these results, we can obtain another exact formula for the number of  monic $n$-smooth  polynomials  of degree $m$ by using  Corollary \ref{Gen2}. This formula is useful when $n$ is small.

\begin{cor}\label{Smooth2} The number of  monic   $n$-smooth polynomials over $\Fq$ with degree $m$ is
\begin{align*}
\sum_{ l_1 +2l_2  + \cdots  +n l_n=m}
N^*(m,\prod_{i=1}^nI_i^{l_i},1)=q^{m}\prod_{i=	1}^n\sum_{l_i\ge 0}{|I_i|+l_i-1\choose l_i}q^{-il_i}\llbracket \sum_{i=1}^n il_i=m\rrbracket.
\end{align*}
\end{cor}
\begin{proof}
A degree $m$ monic polynomial is $n$-smooth if it contains no factors with degree greater than $n$. Hence, summing over all cases with $T=\{1,\ldots,n\}$, where the polynomial is a  product of factors with degrees in $T$ with  Corollary \ref{Gen2}, we obtain the result.\qed
\end{proof}

As a consequence  of  Corollaries \ref{Smooth1} and \ref{Smooth2}, we have an identity for the number of monic $n$-smooth polynomials of degree $m$ over $\fq$.  This number is given by
\begin{align*}
&q^{m}\prod_{i=n+1}^m\sum_{j_i=0}^{|I_i|}q^{-ij_i}{|I_i|\choose j_i}(-1)^{j_i}\llbracket \sum_{i=n+1}^m{ij_i\le m} \rrbracket\\
&=q^{m}\prod_{i=	1}^n\sum_{l_i\ge 0}{|I_i|+l_i-1\choose l_i}q^{-il_i}\llbracket \sum_{i=1}^n il_i=m\rrbracket.
\end{align*}

This identity can be proven in an elementary way. Indeed, by the unique factorization of monic polynomials,  we have 
\begin{align}\label{UN1}
\prod_{i\ge 1}(1-z^i)^{-|I_i|}=\sum_{k\ge 0}q^kz^k. 
\end{align}
Using generating functions, the number of monic $n$-smooth  polynomials of degree $m$  is 
\begin{align*}
[z^m]\prod_{i=1}^n(1-z^i)^{-|I_i|}
&=[z^m]\prod_{i=1}^n\sum_{l_i\ge 0}{|I_i|+l_i-1\choose l_i}z^{il_i}\\
&=\prod_{i=	1}^n\sum_{l_i\ge 0}{|I_i|+l_i-1\choose l_i}\llbracket \sum_{i=1}^n il_i=m\rrbracket.
\end{align*}
Using equation (\ref{UN1}), we have that
\begin{align*}
\prod_{i=1}^n(1-z^i)^{-|I_i|}=\sum_{k\ge 0}q^kz^k\prod_{i\ge n+1}(1-z^i)^{|I_i|}.
\end{align*}
Hence, the number of monic $n$-smooth  polynomials of degree $m$  is also given by
\begin{align*}
[z^m]\sum_{k\ge 0}q^kz^k\prod_{i\ge n+1}(1-z^i)^{|I_i|}
&=[z^m]\sum_{k\ge 0}q^kz^k\prod_{i\ge n+1}\sum_{j_i=0}^{|I_i|}{|I_i|\choose j_i}(-1)^{j_i}z^{ij_i}\\
&=q^{m}\prod_{i=n+1}^m\sum_{j_i=0}^{|I_i|}q^{-ij_i}{|I_i|\choose j_i}(-1)^{j_i}\llbracket \sum_{i=n+1}^m{ij_i\le m}\rrbracket.
\end{align*}

\section{$w=1$: polynomials with prescribed trace term}

In this section, we consider the case where $w=1$. In this case, we are able to obtain the exact formulas for the number of degree $m$ polynomials of the form $f(x)=x^m+\alp x^{m-1}+g(x)$, where $g(x)\in\Fq[x]$ has degree at most $m-2$, $\alp\in\Fq$ is fixed, and $f$ has a prescribed factorization pattern in terms of  degrees of irreducible factors, with or without multiplicity counted. 

In order to do so, we use the formula from Section 3 for the number of degree $m$ monic irreducible polynomials when the second highest degree term is prescribed. More explicitly, the number of degree $i$ irreducible polynomials of the form $f(x)=x^i+\alp x^{i-1}+g(x)$, for $g\in\Fq[x]$ of degree at most  $i-2$ and fixed  $\alp\in\Fq$ is
\begin{align}
|\{\langle f \rangle\in I_i, \langle f \rangle=\langle  x+\alp \rangle\}|=I(i,\langle  x+\alp \rangle)=a_i+b_i\llbracket \alp=0 \rrbracket,
\end{align}
where 
\begin{align}
a_i=\frac{1}{iq}\sum_{p\nmid k \mid i}\mu(k)q^{i/k}, ~~~~&   
b_i=\frac{1}{i}\sum_{p\mid k \mid i}\mu(k)q^{i/k}.\label{b_i}
\end{align}

In general, we can obtain answers that are in terms of $a_i, b_i,$ and $|I_i|$. This turns out to be sufficient for obtaining the known formula for the number of degree $m$ monic polynomials with $r$ distinct roots when the second highest degree term is prescribed. In addition, we obtain an analogue of this formula when the multiplicity of the roots is counted. 

We also obtain formulas for the number of  monic $n$-smooth monic polynomials of degree $m$ when the second highest degree term is prescribed in a similar way to the case $w=0$ and obtain similar looking identities. 

When $w=1$, we have  $G=\{\langle  x+\alp \rangle:\alp\in\Fq\}$, and 
\begin{align*}
\langle  x+\alp \rangle\langle  x+\bet \rangle=\langle x+\alp+\bet\rangle.
\end{align*}

For $\alp\in\Fq$, note that
\begin{align*}
\langle  x+\alp \rangle^k= \langle x+k\alp \rangle.
\end{align*}
Using $\langle x\rangle=1$, it follows that
\begin{align*}
\sum_{\alp\in\Fq}(\langle  x+\alp \rangle)^k&=\sum_{\alp\in\Fq} \langle x+k\alp \rangle=q \langle x \rangle \llbracket p\mid k \rrbracket  + \sum_{\alp\in\Fq} \langle  x+\alp \rangle \llbracket p\nmid k \rrbracket\\
&=q\llbracket p\mid k \rrbracket +qE\llbracket p\nmid k \rrbracket.
\end{align*}
Using $EJ=0$, it follows that
\begin{align*}
J\sum_{\alp\in\Fq}(\langle  x+\alp \rangle)^k=qJ\llbracket p\mid k \rrbracket.
\end{align*}
For $k\ge1$ using $J^k=J$, we have
\begin{align}\label{Setup}
J^k\sum_{\alp\in\Fq}(\langle  x+\alp \rangle)^k=J^kq\llbracket p\mid k \rrbracket.
\end{align}

Using this formula, we obtain some facts which are useful for deriving results for $N$ and $N^*$ when $w=1$.

\begin{prop}\label{fact} We have the following facts:
\begin{itemize}
\item[(i)] $J\prod_{\alp\in\Fq}(1+\langle  x+\alp  \rangle y)=J(1-(-y)^p)^\frac{q}{p}$,
\item[(ii)] $J\prod_{\alp\in\Fq}\frac{1}{1-\langle  x+\alp  \rangle y}=J\left(\frac{1}{1-y^p}\right)^\frac{q}{p}$.
\end{itemize}
\end{prop}
\begin{proof}
1) Using  Proposition \ref{PowerProp1}, Equation (\ref{Setup}), and the power series definition for $\exp$ and $\log$, we have
\begin{align*}
J\prod_{\alp\in\Fq}(1+\langle  x+\alp  \rangle y)&=J\prod_{\alp\in\Fq}(1+\langle  x+\alp  \rangle Jy)\\
&=J\exp\left(\sum_{\alp\in\Fq}\ln(1+\langle  x+\alp  \rangle Jy)\right)\\
&=J\exp\left(\sum_{\alp\in\Fq}\sum_{k\ge 1}(-1)^{k-1}\frac{\langle  x+\alp \rangle^kJ^ky^k}{k}\right)\\
&=J\exp\left(\sum_{k\ge 1}(-1)^{k-1} \frac{J^ky^k}{k}( q J
 \llbracket p\mid k \rrbracket) \right)\\
&=J\exp\left(\sum_{k\ge 1}(-1)^{k-1} \frac{y^k}{k}( q
 \llbracket p\mid k \rrbracket)\right)\\
&=J\exp\left(\frac{-q}{p}\sum_{k\ge 1} \frac{(-y)^{pk}}{k}\right)\\
&=J\exp\left(\frac{q}{p}\ln(1-(-y)^p)\right)\\
&= J(1-(-y)^p)^\frac{q}{p}.
\end{align*}
2) Similarly, we have
\begin{align*}
J\prod_{\alp\in\Fq}\frac{1}{1-\langle  x+\alp  \rangle y}&=J\prod_{\alp\in\Fq}\frac{1}{1-\langle  x+\alp  \rangle Jy}\\&=J\exp\left(\sum_{\alp\in\Fq}\ln\left(\frac{1}{1-\langle  x+\alp  \rangle Jy}
\right)\right)\\
&=J\exp\left(\sum_{\alp\in\Fq}\sum_{k\ge 1}\frac{\langle  x+\alp \rangle^kJ^ky^k}{k}\right)\\
&=J\exp\left(\sum_{k\ge 1} \frac{J^ky^k}{k}(q J
 \llbracket p\mid k \rrbracket) \right)\\
&=J\exp\left(\frac{q}{p}\sum_{k\ge 1} \frac{y^{pk}}{k}\right)\\
&=J\exp\left(\frac{q}{p}\ln\left(\frac{1}{1-y^p}\right)\right)\\
&=J\left(\frac{1}{1-y^p}\right)^\frac{q}{p}.\qed
\end{align*}
\end{proof}
Define the following numbers
\begin{align}
A_m(a,0)&={aq/p \choose m/p}(-1)^{m+m/p}\llbracket p\mid m \rrbracket,\label{Simple1}\\
B_m(a,0)&={aq/p+m/p-1\choose m/p} \llbracket p\mid m \rrbracket, \label{Simple2}
\end{align}
and,  for $b\neq 0$, 
\begin{align}
A_m(a,b)&=\sum_{j=0}^{\lfloor m/p \rfloor}{aq/p\choose j}{b\choose m-pj}(-1)^{j+pj}, \label{defA}\\ 
B_m(a,b)&=\sum_{j=0}^{\lfloor m/p \rfloor}{aq/p+j-1\choose j}{b+m-pj-1\choose m-pj}. \label{defB}
\end{align}

Combining these numbers with (\ref{fact}), we obtain information related to the set of monic  irreducible polynomials of degree $i$.

\begin{prop}\label{Important} Let $a_i, b_i$ defined in (\ref{b_i}) and $A_m(a_i, b_i)$, $B_m(a_i, b_i)$ be defined in (\ref{defA}) and (\ref{defB}).  Then
\begin{itemize}
\item[(i)]   $J\prod_{f\in I_i}(1+\langle f \rangle  y)=J\sum_{m\ge 0}A_m(a_i,b_i)y^m. $
\item[(ii)]  $J\prod_{f\in I_i}\frac{1}{1-\langle f \rangle  y}=J\sum_{m\ge 0}B_m(a_i,b_i)y^m$. 
\end{itemize}
\end{prop}
\begin{proof} 1) Using $\langle x \rangle=1$ and $I(i,\langle  x+\alp \rangle)=a_i+b_i\llbracket \alp=0 \rrbracket$, $J^k=J$ for $k\ge 1$, and Proposition \ref{fact}, we have
\begin{align*}
J\prod_{f\in I_i}(1+\langle f \rangle  y)&=J\prod_{\alp\in \Fq}(1+\langle  x+\alp  \rangle y)^{a_i+b_i\llbracket\alp=0\rrbracket}\\
&=J  (1+\langle x\rangle y)^{b_i} \prod_{\alp\in \Fq}(1+\langle  x+\alp  \rangle y)^{a_i}\\
&=J(1-(-y)^p)^\frac{a_iq}{p}(1+y)^{b_i}\\
&=J\sum_{m\ge 0}\sum_{j=0}^{\lfloor m/p \rfloor}{a_iq/p\choose j}(-1)^j(-1)^{pj}{b_i\choose m-pj}y^m\\
&=J\sum_{m\ge 0}\sum_{j=0}^{\lfloor m/p \rfloor}{a_iq/p\choose j}{b_i\choose m-pj}(-1)^{j+pj}y^m\\
&=J\sum_{m\ge 0}A_m(a_i,b_i)y^m.
\end{align*}
2) Similarly, we have
\begin{align*}
J\prod_{f\in I_i}\frac{1}{1-\langle f \rangle  y}&=J\prod_{\alp\in \Fq}\left(\frac{1}{1-\langle  x+\alp  \rangle y}\right)^{a_i+b_i\llbracket\alp=0\rrbracket}\\
&=J \left(\frac{1}{1-\langle x \rangle y} \right)^{b_i}  \prod_{\alp\in \Fq}\left(\frac{1}{1-\langle  x+\alp  \rangle y}\right)^{a_i}\\ 
&=J\left(\frac{1}{1-y^p}\right)^\frac{a_iq}{p}\left(\frac{1}{1-y}\right)^{b_i}\\
&=J\sum_{m\ge 0}\sum_{j=0}^{\lfloor m/p \rfloor}{a_iq/p+j-1\choose j}{b_i+m-pj-1\choose m-pj}y^m\\
&=J\sum_{m\ge 0}B_m(a_i,b_i)y^m.
\qed
\end{align*}
\end{proof}

Using Proposition \ref{Important}, we can obtain formulas for $N$ and $N^*$ when $w=1$.

\begin{thm}\label{Second1} Suppose $w=1$. Then 

\begin{align*}
&N(m,\prod_{i\in T}I_{i}^{r_i},\langle  x+\alp \rangle)\\
&=q^{m-1}\prod_{i\in T}{|I_i|\choose r_i}q^{-ir_i}\sum_{j_i=0}^{|I_i|-r_i}q^{-ij_i}{|I_i|-r_i\choose j_i}(-1)^{j_i}\llbracket \sum_{i\in T}{i(r_i+j_i)\le m} \rrbracket\\
&+\frac{v(\alp)}{q}\prod_{i\in T}
\sum_{k_i\ge 0}A_{k_i}(a_i,b_i){k_i\choose r_i}(-1)^{k_i-r_i}\llbracket \sum_{i=1}^n ik_i=m\rrbracket,
\end{align*}
where $v(\alp)=q\llbracket \alp=0 \rrbracket -1$.
\begin{align*}
\end{align*}
\end{thm}
\begin{proof}
Using Theorem \ref{Main1}, we have
\begin{align*}
&N(m,\prod_{i\in T}I_i^{r_i},\langle  x+\alp \rangle)\\&=q^{m-1}\prod_{i\in T}{|I_i|\choose r_i}q^{-ir_i}\sum_{j_i=0}^{|I_i|-r_i}q^{-ij_i}{|I_i|-r_i\choose j_i}(-1)^{j_i}\llbracket \sum_{i\in T}{i(r_i+j_i)\le m} \rrbracket\\
&+[\langle  x+\alp  \rangle z^m\prod_{i\in T}u_i^{r_i}]J\prod_{i\in T}\prod_{g\in I_i}(1+\langle g  \rangle z ^{i}(u_i-1)).
\end{align*}
Using $J^k=J$ for any positive integer $k$ and Proposition \ref{Important}, we have
\begin{align*}
&[\langle  x+\alp  \rangle z^m\prod_{i\in T}u_i^{r_i}]J\prod_{i\in T}\prod_{g\in I_i}(1+\langle g  \rangle z ^{i}(u_i-1))\\
&=[\langle  x+\alp  \rangle z^m\prod_{i\in T}u_i^{r_i}]J\prod_{i\in T}\sum_{k_i\ge 0}A_{k_i}(a_i,b_i)z^{ik_i}(u_i-1)^{k_i}.\\
\end{align*}
Extracting coefficients of the $u_i$  by using the  Binomial Theorem,  the latter equals
\begin{align*}
[\langle  x+\alp  \rangle z^m]J\prod_{i\in T}
\sum_{k_i\ge 0}A_{k_i}(a_i,b_i)z^{ik_i}{k_i\choose r_i}(-1)^{k_i-r_i}. 
\end{align*}
Extracting the coefficient of $z$,  we obtain
\begin{align*}
[\langle  x+\alp \rangle]J\prod_{i\in T}
\sum_{k_i\ge 0}A_{k_i}(a_i,b_i){{k_i}\choose r_i}(-1)^{{k_i}-r_i}\llbracket \sum_{i=1}^n ik_i=m\rrbracket.
\end{align*}
Using $J=1-E=\sum_{\alp\in\Fq} \frac{v(\alp)}{q} \langle  x+\alp \rangle,$ and extracting $\langle  x+\alp \rangle$,  this simplifies to
\begin{align*}
&\frac{v(\alp)}{q}\prod_{i\in T}
\sum_{k_i\ge 0}A_{k_i}(a_i,b_i){k_i\choose r_i}(-1)^{k_i-r_i}\llbracket \sum_{i=1}^n ik_i=m\rrbracket,
\end{align*}
The proof is complete by combining all the pieces together. 
\qed\end{proof}

In the following special case, we obtain a simpler result.

\begin{cor}\label{SecondCor1} Suppose $w=1$. Suppose that $p\nmid i$ for each $i \in T$. 
\\
If $p\nmid m$, then  
\begin{align*}
&N(m,\prod_{i\in T}I_i^{r_i},\langle  x+\alp \rangle)\\
&=q^{m-1}\prod_{i\in T}{|I_i|\choose r_i}q^{-ir_i}\sum_{j_i=0}^{|I_i|-r_i}q^{-ij_i}{|I_i|-r_i\choose j_i}(-1)^{j_i}\llbracket \sum_{i\in T}{i(r_i+j_i)\le m} \rrbracket.
\end{align*}
If $p\mid m$, then
\begin{align*}
&N(m,\prod_{i\in T}I_i^{r_i},\langle  x+\alp \rangle)\\
&=q^{m-1}\prod_{i\in T}{|I_i|\choose r_i}q^{-ir_i}\sum_{j_i=0}^{|I_i|-r_i}q^{-ij_i}{|I_i|-r_i\choose j_i}(-1)^{j_i}\llbracket \sum_{i\in T}{i(r_i+j_i)\le m} \rrbracket\\
&+\frac{v(\alp)}{q}\prod_{i\in T}
\sum_{k_i=0}^{|I_i|/p}{|I_i|/p \choose k_i}{pk_i\choose r_i}(-1)^{k_i-r_i}\llbracket \sum_{i=1}^n ik_i=m/p\rrbracket,
\end{align*}
where $v(\alp)=q\llbracket \alp=0 \rrbracket -1$.
\end{cor}
\begin{proof}
The result follows from Theorem \ref{Second1} by setting $b_i=0$ for each $i$, and using (\ref{Simple1}), and noting that $qa_i=|I_i|$ for $p\nmid i$.\qed
\end{proof}

Setting $T=\{1\}$, we obtain the known result for the number of monic polynomials with a given number of linear factors when the  trace term is fixed.

\begin{cor}[Theorem 3.1 \cite{Zhou}] The number of monic polynomials over $\Fq$ of the form $x^m+\alp x^{m-1}+g(x)$ for fixed $\alp\in\Fq$,  where $g\in\Fq[x]$ has degree at most $m-2$, that have $r$ distinct linear factors is given as follows:

If $p\nmid m$, then
\begin{align*}
N(m,I_1^{r},\langle  x+\alp \rangle)
&=q^{m-r-1}{q\choose r}\sum_{j=0}^{m-r}q^{-j}{q-r\choose j}(-1)^{j}.
\end{align*}
If $p\mid m$, then
\begin{align*}
N(m,I_1^{r},\langle  x+\alp \rangle)=&q^{m-r-1}{q\choose r}\sum_{j=0}^{m-r}q^{-j}{q-r\choose j}(-1)^{j}\\
&+\frac{v(\alp)}{q}
{q/p \choose m/p}{m\choose r}(-1)^{m/p-r},
\end{align*}
where $v(\alp)=q\llbracket \alp=0 \rrbracket-1.$
\end{cor}
\begin{proof}
The result follows from Theorem \ref{Second1} by taking $T=\{1\}$, and using $|I_1|=q$.
\qed\end{proof}

Now, we state a result about  degree $m$  monic $n$-smooth polynomials with a prescribed trace coefficient, which comes from Theorem \ref{Second1}. This formula is most useful when $n$ is close to $m$.

\begin{cor}\label{Smooth3} The number of monic $n$-smooth polynomials over $\Fq$ of the form $x^m+\alp x^{m-1}+g(x)$ for fixed $\alp\in\Fq$,  where $g\in\Fq[x]$ has degree at most $m-2$ is
\begin{align*}
N(m,\prod_{i=n+1}^mI_{i}^{0},\langle  x+\alp \rangle)
&=q^{m-1}\prod_{i=n+1}^mq^{-ir_i}\sum_{j_i=0}^{|I_i|}q^{-ij_i}{|I_i|\choose j_i}(-1)^{j_i}\llbracket \sum_{i=n+1}^m{ij_i\le m} \rrbracket\\
&+\frac{v(\alp)}{q}\prod_{i=n+1}^m
\sum_{k_i\ge 0}A_{k_i}(a_i,b_i)(-1)^{k_i}\llbracket \sum_{i=1}^n ik_i=m\rrbracket, 
\end{align*}
where $v(\alp)=q\llbracket \alp=0 \rrbracket -1.$
\end{cor}
\begin{proof}
The result follows from Theorem \ref{Second1} by setting $T=\{n+1,\ldots,m\}$ and using the fact that a monic polynomial is $n$-smooth if it has no irreducible factors with degree larger than $n$. \qed
\end{proof}

Next, we give a general formula for $N^*$. That is the case where the multiplicity of the the factors is counted. 
\begin{thm}\label{Second2}
Suppose $w=1$. Then

\begin{align*}
&N^*(m,\prod_{i\in T}I_i^{l_i},\langle  x+\alp \rangle)\\
&=q^{m-1}\prod_{i\in T}{|I_i|+l_i-1\choose l_i}q^{-il_i}\sum_{j_i=0}^{|I_i|}{|I_i|\choose j_i}(-1)^{j_i}q^{-ij_i}
\llbracket \sum_{i\in T} i(l_i+j_i) \le m\rrbracket\\
&+\frac{v(\alp)}{q}\prod_{i\in T}B_{l_i}(a_i,b_i)\sum_{k_i\ge 0}A_{k_i}(a_i,b_i)(-1)^{k_i}
\llbracket \sum_{i\in T} i(l_i+k_i)=m \rrbracket,
\end{align*}
where $v(\alp)=q\llbracket \alp=0 \rrbracket -1.$
\end{thm}

\begin{proof}
Using Theorem \ref{Main2}, we have
\begin{align*}
&N^*(m,\prod_{i\in T}I_i^{l_i},\langle  x+\alp \rangle)\\&=q^{m-1}\prod_{i\in T}{|I_i|+l_i-1\choose l_i}q^{-il_i}\sum_{j_i=0}^{|I_i|}{|I_i|\choose j_i}(-1)^{j_i}q^{-ij_i}
\llbracket \sum_{i\in T} i(l_i+j_i) \le m\rrbracket\\
&+[<a+\alp>z^m\prod_{i\in T}u_i^{l_i}]J\prod_{i\in T}\prod_{g\in I_i}\left(
\frac{1-\langle g  \rangle z ^{i}}{1-\langle g  \rangle z ^{i}u_i}\right).
\end{align*}
Using $J^k=J$ for any positive integer $k$,  applying Proposition \ref{Important}, and extracting coefficients of the $u_i$, $z$, and then $\langle  x+\alp \rangle$, we have
\begin{align*}
&[\langle  x+\alp  \rangle z^m\prod_{i\in T}u_i^{l_i}]J\prod_{i\in T}\prod_{g\in I_i}\left(
\frac{1-\langle g  \rangle z ^{i}}{1-\langle g  \rangle z ^{i}u_i}\right)\\
&=[\langle  x+\alp  \rangle z^m] J\prod_{i\in T}B_{l_i}(a_i,b_i)z^{il_i}\sum_{k_i\ge 0}A_{k_i}(a_i,b_i)(-z^{i})^{k_i}\\
&=[\langle  x+\alp  \rangle z^m] J\prod_{i\in T}B_{l_i}(a_i,b_i)\sum_{k_i\ge 0}A_{k_i}(a_i,b_i)(-1)^{k_i}z^{i(l_i+k_i)}\\
&=[\langle  x+\alp \rangle] J\prod_{i\in T}B_{l_i}(a_i,b_i)\sum_{k_i\ge 0}A_{k_i}(a_i,b_i)(-1)^{k_i}
\llbracket \sum_{i\in T} i(l_i+k_i)=m \rrbracket\\
&=\frac{v(\alp)}{q}\prod_{i\in T}B_{l_i}(a_i,b_i)\sum_{k_i\ge 0}A_{k_i}(a_i,b_i)(-1)^{k_i}
\llbracket \sum_{i\in T} i(l_i+k_i)=m \rrbracket,
\end{align*}
since $J=1-E=\sum_{\alp\in\Fq}\frac{v(\alp)}{q} \langle  x+\alp \rangle.$ Hence, the result follows.
\qed\end{proof}

As a corollary to Theorem \ref{Second2}, we have the following simpler result.

\begin{cor}\label{SecondCor2} Suppose $w=1$. Suppose that $p\nmid i$ for each $i \in T$.  \\
If $p\nmid m$ or $p\nmid l_i$ for some $i$, then
\begin{align*}
&N^*(m,\prod_{i\in T}I_i^{l_i},\langle  x+\alp \rangle)\\
&=q^{m-1}\prod_{i\in T}{|I_i|+l_i-1\choose l_i}q^{-il_i}\sum_{j_i=0}^{|I_i|}{|I_i|\choose j_i}(-1)^{j_i}q^{-ij_i}
\llbracket \sum_{i\in T} i(l_i+j_i) \le m\rrbracket.
\end{align*}
If $p\mid m$ and $p\mid l_i$ for each $i$, then
\begin{align*}
&N^*(m,\prod_{i\in T}I_i^{l_i},\langle  x+\alp \rangle)\\
&=q^{m-1}\prod_{i\in T}{|I_i|+l_i-1\choose l_i}q^{-il_i}\sum_{j_i=0}^{|I_i|}{|I_i|\choose j_i}(-1)^{j_i}q^{-ij_i}
\llbracket \sum_{i\in T} i(l_i+j_i) \le m\rrbracket\\
&+\frac{v(\alp)}{q}\prod_{i\in T}{|I_i|/p+l_i/p-1\choose l_i/p}\sum_{k_i=0}^{|I_i|/p}{|I_i|/p\choose k_i}(-1)^{pk_i}
\llbracket \sum_{i\in T} i(l_i/p+k_i)=m/p \rrbracket,
\end{align*}
where $v(\alp)=q\llbracket \alp=0 \rrbracket -1.$
\end{cor}

\begin{proof} The result follows from Theorem \ref{Second2} by setting $b_i=0$ for each $i$,  using (\ref{Simple1}) and (\ref{Simple2}), and noting that $qa_i=|I_i|$ for $p\nmid i$.\qed
\end{proof}

From this corollary, we can obtain the number of degree $m$ monic polynomials $f(x)$ with a given number of roots counting multiplicity, with a fixed coefficient of $x^{m-1}$.

\begin{cor} The number of monic polynomials over $\Fq$ of the form $x^m+\alp x^{m-1}+g(x)$ for fixed $\alp\in\Fq$, where $g(x)\in\Fq[x]$ has degree  at most $m-2$ that have $l$ linear factors counting multiplicity is given as follows:

If $p\nmid m$ or $p\nmid l$, then
\begin{align*}
N^*(m,\prod_{i\in T}I_1^l,\langle  x+\alp \rangle)
&=q^{m-l-1}{q+l-1\choose l}\sum_{j=0}^{m-l}{q\choose j}(-1)^{j}q^{-j}.
\end{align*}
If $p\mid m$ and $p\mid l$, then
\begin{align*}
N^*(m,\prod_{i\in T}I_1^l,\langle  x+\alp \rangle)
&=q^{m-l-1}{q+l-1\choose l}\sum_{j=0}^{m-l}{q\choose j}(-1)^{j}q^{-j}\\
&+\frac{v(\alp)}{q}{q/p+l/p-1\choose l/p}{q/p\choose (m-l)/p}(-1)^{m-l},
\end{align*}
where $v(\alp)=q\llbracket \alp=0 \rrbracket -1.$
\end{cor}
\begin{proof}
The result follows from Corollary \ref{SecondCor2} by taking $T=\{1\}$, and using $|I_1|=q$.
\qed\end{proof}

Using Corollary \ref{Second2}, we obtain another formula for the number of $n$-smooth degree $m$ monic polynomials with a prescribed trace coefficient.  The result looks different from Theorem \ref{Smooth3} and it is useful when $n$ is small. 

\begin{cor}\label{Smooth4} The number of monic $n$-smooth polynomials over $\Fq$ of the form $x^m+\alp x^{m-1}+g(x)$ for fixed $\alp\in\Fq$,  where $g\in\Fq[x]$ has degree at most $m-2$ is
\begin{align*}
\sum_{ l_1+2l_2+\cdots + nl_n=m}
N^*(m,\prod_{i=1}^nI_i^{l_i},1)=q^{m}\prod_{i=	1}^n\sum_{l_i\ge 0}{|I_i|+l_i-1\choose l_i}q^{-il_i}\llbracket \sum_{i=1}^n il_i=m \rrbracket
&\\+\frac{v(\alp)}{q}\prod_{i=1}^n\sum_{l_i\ge 0}B_{l_i}(a_i,b_i)
\llbracket \sum_{i=1}^n il_i=m \rrbracket,
\end{align*}
where $v(\alp)=q\llbracket \alp=0 \rrbracket -1.$
\end{cor}
\begin{proof}
A degree $m$ monic polynomial is $n$-smooth if it contains no factors above degree $n$. Hence, summing over all cases with $T=\{1,\ldots,n\}$ with Theorem \ref{Second2}, where the polynomial is a  product  of factors with degrees in $T$, we obtain the result. 
\qed
\end{proof}

Corollaries \ref{Smooth3} and \ref{Smooth4} give two different looking expressions for the number of degree $m$ monic $n$-smooth polynomials of the form $x^m+\alp x^{m-1}+g(x)$, where $\alp\in\Fq$ is fixed and $g(x)\in\Fq[x]$ has degree at most $m-2$. In order to use these formulas in computation, one should check the size of $n$ compared to $m$. If $n$ is much bigger than $1$, it is likely faster to use Corollary \ref{Smooth3}, while if $n$ is much closer to $1$ than $m$, it is probably faster to use Corollary \ref{Smooth4}.

The equivalence of these two formulas can be verified directly using generating functions in a similar way to the case $w=0$.
Indeed,  the number of   $n$-smooth polynomials of degree $m$ of the form $x^m+\alp x^{m-1}+g(x)$, where $\alp\in\Fq$ is fixed and $g(x)\in\Fq[x]$ has degree at most $m-2$, is given by 
\begin{align*}
&[\langle x+\alp\rangle z^m]\prod_{i=1}^n \prod_{f\in I_i}(1-\langle f\rangle z^i)^{-1}.
\end{align*}
From unique factorization of monic polynomials, we have
\begin{align}\label{UN2}
F(z)=\prod_{i\ge 1}\prod_{f\in I_i}(1-\langle f\rangle z^i)^{-1}=1+\sum_{k\ge 1}\sum_{f\in M_k}\langle f \rangle z^k.
\end{align}
Using $E\langle f \rangle=E$ and $E^2=E$, we have
\begin{align}\label{eqE}
EF(z)=E\prod_{i\ge 1}(1-z^i)^{-|I_i|}=E\sum_{k\ge 0}q^kz^k.
\end{align}
Using $J\sum_{f\in M_k}\langle f \rangle=0$ for $k\ge w=1$, we have
\begin{align}\label{eqJ}
JF(z)=J\prod_{i\ge 1}\prod_{f\in I_i}(1-\langle f\rangle z^i)^{-1}=J.
\end{align}
Applying $F(z)=\prod_{i\ge 1} \prod_{f\in I_i}(1-\langle f\rangle z^i)^{-1}$ to Equations (\ref{eqE}) and (\ref{eqJ}), we obtain
\begin{align*}
&E\prod_{i=1}^n \prod_{f\in I_i}(1-\langle f\rangle z^i)^{-1}=E\prod_{i=1}^n(1-z^i)^{-|I_i|}=E\sum_{k\ge 0}q^kz^k\prod_{i\ge n+1}(1-z^i)^{|I_i|},\\
&J\prod_{i=1}^n \prod_{f\in I_i}(1-\langle f\rangle z^i)^{-1}=J\prod_{i\ge n+1} \prod_{f\in I_i}(1-\langle f\rangle z^i).
\end{align*}
From Proposition \ref{Important}, we have
\begin{align*}
J\prod_{i=1}^n \prod_{f\in I_i}(1-\langle f\rangle z^i)^{-1}=J\prod_{i\ge 1}\sum_{k\ge 0}B_k(a_i,b_i)z^{ik},\\ 
J\prod_{i\ge n+1}\prod_{f\in I_i}(1-\langle f\rangle z^i) =\prod_{i\ge n+1}\sum_{k\ge 0}A_k(a_i,b_i)(-1)^kz^{ik}.
\end{align*}
Hence, using $E+J=1$, we obtain
\begin{align*}
&\prod_{i=1}^n \prod_{f\in I_i}(1-\langle f\rangle z^i)^{-1}\\&=E\prod_{i=1}^n(1-z^i)^{-|I_i|}+J\prod_{i\ge 1}\sum_{k\ge 0}B_k(a_i,b_i)z^{ik}\\
&=E\sum_{k\ge 0}q^kz^k\prod_{i\ge n+1}(1-z^i)^{|I_i|}+J\prod_{i\ge n+1}\sum_{k\ge 0}A_k(a_i,b_i)(-1)^kz^{ik}.
\end{align*}
Using $E=\frac{1}{q}\sum_{\alp\in\Fq}\langle x+\alp \rangle$ and $J=1-E=\frac{1}{q}\sum_{\alp\in\Fq}v(\alp)\langle x+\alp \rangle$, we obtain
\begin{align*}
&[\langle x+\alp\rangle z^m]\prod_{i=1}^n\prod_{f\in I_i}(1-\langle f \rangle z^i)^{-1}\\
&=\frac{1}{q}\prod_{i=	1}^n\sum_{l_i\ge 0}{|I_i|+l_i-1\choose l_i}\llbracket \sum_{i=1}^n il_i=m \rrbracket
\\&+\frac{v(\alp)}{q}\prod_{i=1}^n\sum_{l_i\ge 0}B_{l_i}(a_i,b_i)
\llbracket \sum_{i=1}^n il_i=m \rrbracket\\
&=q^{m-1}\prod_{i=n+1}^m\sum_{j_i=0}^{|I_i|}q^{-ij_i}{|I_i|\choose j_i}(-1)^{j_i}\llbracket \sum_{i=n+1}^m{ij_i\le m} \rrbracket\\
&+\frac{v(\alp)}{q}\prod_{i=n+1}^m
\sum_{k_i\ge 0}A_{k_i}(a_i,b_i)(-1)^{k_i}\llbracket \sum_{i=n+1}^m ik_i=m\rrbracket,
\end{align*}
as desired.

\section*{Acknowledgment}
We thank Zhicheng Gao for many helpful discussions.


\begin{thebibliography}{999}



\bibitem{Car}
M. Car, Th\'{e}or\`{e}mes de densit\'{e} dans $\fq[X]$. (French)  (Density theorems in $
\fq[X]$),  {\em Acta Arith.} 48 (1987), no. 2, 145-165.



\bibitem{car}
L. Carlitz,
A theorem of Dickson on irreducible polynomials,
{\em Proc. Amer. Math. Soc.} { 3} (1952), 693-700.




\bibitem{DGP09}
L. Dong, Z. Gao, D. Panario, Enumeration of decomposable combinatorial structures with restricted patterns.  {\em Ann. Comb.} 12 (2009), no. 4, 357-372.

\bibitem{fityuc}
R.~W.~Fitzgerald and J.~L.~Yucas,
Irreducible polynomials over GF(2) with three prescribed
coefficients, {\em  Finite Fields Appl.} {9} (2003), 286-299.

\bibitem{Paper}
Z. Gao, S. Kuttner, Q. Wang, On enumeration of irreducible polynomials and related objects over a finite field with respect to their trace and norm,  {\em  Finite Fields Appl.} 69 (2021), 101770, 25pp. 

\bibitem{GKW21}
Z. Gao, S. Kuttner, Q. Wang, Counting irreducible polynomials with prescribed coefficients over a finite field, preprint.

\bibitem{GP06}
Z. Gao, D. Panario,  Degree distribution of the greatest common divisor of polynomials over $\mathbb{F}_q$. 
{\em Random Structures Algorithms} 29 (2006), no. 1, 26-37.


\bibitem{Knop}
A. Knopmacher, J. Knopmacher,  Counting polynomials with a given number of zeros in a finite field, {\em Linear Multilinear Algebra} 26 (1990), 267-292.

\bibitem{Knop2}
A. Knopmacher, J. Knopmacher, Counting irreducible factors of polynomials over a finite field, {\em Discrete Math.} 112 (1993), 103-118.

\bibitem{Li}
J. Li, D. Wan, Distance distribution  in Reed-Solomon codes,  {\em IEEE Trans. Inform. Theory} 66 (2020), no. 5, 2743-2750. 



\bibitem{Finite Fields}
R. Lidl, H. Niederreiter, {\em Finite Fields}, Cambridge University Press, Cambridge, 1997.


\bibitem{Lovorn}
R. Lovorn,  Rigourous, subexponential algorithms for discrete logarithm algorithms
in $\F_{p^2}$. PhD thesis~ University of Georgia, 1992. 



\bibitem{MoisioRanto}
M.~Moisio and K.~Ranto,
Elliptic curves and explicit enumeration of irreducible
polynomials with two coefficients prescribed,
{\em Finite Fields Appl.} {14} (2008),
798-815.


\bibitem{Handbook}
G. L. Mullen, D. Panario,
{\em Handbook of Finite Fields}, Discrete Mathematics and its Applications (Boca Raton). CRC Press, Boca Raton, FL, 2013.


\bibitem{Odlyzko}
A. Odlyzko,  Discrete logarithms and their cryptographic significance. In Advances in Cryptology, Proceedings of Eurocrypt 1984 (1985), vol. 209 of Lecture Notes in Computer Science, Springer-Verlag, pp. 224-314. 


\bibitem{Panarioetal}
D. Panario, X. Gourdon, P.  Flajolet,  An analytic approach to smooth polynomials over finite fields. Algorithmic number theory (Portland, OR, 1998), 226-236, Lecture Notes in Comput. Sci., 1423, Springer, Berlin, 1998.

\bibitem{Ruskey}
F. Ruskey, C. R.  Miers, J.  Sawada, 
The number of irreducible polynomials and Lyndon words with given trace. 
{\em SIAM J. Discrete Math.} 14 (2001), no. 2, 240-245.

\bibitem{SF96}
R. Sedgewick and P. Flajolet, An Introduction to the Analysis of Algorithms, Addison Wesley, 1996.

\bibitem{trace}
J.~L.~Yucas,
Irreducible polynomials over finite fields with prescribed
trace/prescribed constant term, {\em Finite Fields Appl.}  {12} (2006), 211-221.


\bibitem{Zhou}
H. Zhou, L. Wang,  W. Wang, Counting polynomials with distinct zeros in finite fields, {\em J.  Number Theory} 174 (2017), 118-135.


\end{thebibliography}
\end{document}